\documentclass[10pt,reqno]{amsart}
\usepackage{amsmath}
\usepackage{amsfonts}
\usepackage{mathrsfs}
\usepackage{amssymb}
\usepackage{amsthm}
\usepackage{enumitem}
\usepackage{bbm}
\usepackage{times}
\usepackage{tabularx}
\usepackage{amsbsy}
\usepackage{mathtools}
\usepackage{tikz}
\usetikzlibrary{arrows,decorations.markings}
\usetikzlibrary{matrix}
\usetikzlibrary{graphs}
\usetikzlibrary{backgrounds}
\usepackage{setspace}     \spacing{1}
\usepackage{color}
\usepackage[normalem]{ulem}
\usepackage{soul}
\usepackage{cancel}

\usepackage[colorlinks=true,linkcolor=blue,citecolor=blue]{hyperref}

\theoremstyle{Theorem}
\newtheorem{theorem}{Theorem} [section]

\newtheorem{proposition}[theorem]{Proposition}
\newtheorem{claim}[theorem]{Claim}
\newtheorem{lemma}[theorem]{Lemma}
\newtheorem{corollary}[theorem]{Corollary}
\newtheorem{custheorem}{Theorem}

\theoremstyle{definition}
\newtheorem{definition}[theorem]{Definition}

\newtheorem{remark}[theorem]{Remark}

\theoremstyle{remark}

\newlist{enumlemma}{enumerate}{3}
\setlist[enumlemma]{label*={(\alph*)}, ref= {(\alph*)} }

\usepackage{marginnote}

\newcommand{\restrict}[2]{{#1}{|_{{ #2}}}}
\newcommand{\diff}{\mathrm{Diff}}

\renewcommand{\epsilon}{\varepsilon}
\renewcommand{\emptyset}{\varnothing}

\def \diag{\mathrm{diag}}

\newcommand{\td}{\tilde}
\newcommand{\wtd}{\widetilde}
\def\Fib{\mathrm{Fiber}}

\DeclareMathOperator{\Diff}{Diff}

\DeclareMathOperator{\Id}{Id}

\newcommand{\sm}{\smallsetminus}
\newcommand{\R}{\mathbb {R}}
\newcommand{\Q}{\mathbb {Q}}
\newcommand{\Z}{\mathbb {Z}}
\newcommand{\N}{\mathbb {N}}

\newcommand{\e}{\epsilon}
\newcommand{\Xt}{X_{\mathrm{thick}}}

\def \sl{\mathfrak{sl}}

\newcommand{\Sl}{\mathrm{SL}}

\newcommand{\So}{\mathrm{SO}}
\def\SO{\So}
\def\SL{\Sl}
\newcommand{\inv}{^{-1}}
\newcommand{\id}{\mathrm{Id}}

\def\calI{\mathcal I}
\def\calA{\mathcal A}
\def\calE{\mathcal E}

\def\calF{\mathcal F}

\def\orb{\mathcal O}

 \newcommand{\lieg}{\mathfrak g}
\newcommand{\lieh}{\mathfrak h}
\newcommand{\liek}{\mathfrak k}

\newcommand{\lien}{\mathfrak n}
\newcommand{\liea}{\mathfrak a}

\newcommand{\liep}{\mathfrak p}

\renewcommand\P{\mathbb{P}}

\def\calM{\mathcal M}
\def\calI{\mathcal I}

\def\Folner{F{\o}lner }

\def\blue{}

\title{Zimmer's conjecture for actions of $\Sl(m,\Z)$}

\author[A.~Brown]{Aaron Brown}
\address{University of Chicago, Chicago, IL 60637, USA}
\email{awb@uchicago.edu}

\author[D.~Fisher]{David Fisher}
\address{Indiana University, Bloomington, Bloomington, IN 47401, USA}
\email{fisherdm@indiana.edu}

\author[S.~Hurtado]{Sebastian Hurtado}
\address{University of Chicago, Chicago, IL 60637, USA}
\email{shurtados@uchicago.edu}

\thanks{DF was partially supported by NSF Grants DMS-1308291 and DMS-1607041.  DF was also partially supported
by the University of Chicago, and by NSF grants DMS 1107452, 1107263, 1107367, ``RNMS: Geometric Structures and Representation Varieties" (the GEAR Network) during a visit to the Isaac Newton Institute in Cambridge.}

\long\def\symbolfootnote[#1]#2{\begingroup\def\thefootnote{\fnsymbol{footnote}}
\footnote[#1]{#2}\endgroup}

\begin{document}
\maketitle

\begin{abstract}
We prove Zimmer's conjecture for $C^2$ actions by finite-index subgroups of $\Sl(m,\Z)$ provided $m>3$. The method utilizes many ingredients from our earlier proof of the conjecture for actions by cocompact lattices in $\Sl(m,\R)$  \cite{BFH} but new ideas are needed to overcome the lack of compactness
of the  space $(G \times M)/\Gamma$ (admitting the induced $G$-action).  Non-compactness allows both measures and Lyapunov exponents to escape to infinity under averaging and
a  number of algebraic, geometric, and dynamical tools are used control this escape.
 New ideas are provided by the
work of Lubotzky, Mozes, and Raghunathan on the structure of nonuniform lattices and, in particular, of $\Sl(m,\Z)$ providing a   geometric decomposition of the cusp into rank one directions, whose  geometry is more easily controlled. The proof also makes use of a precise quantitative form of non-divergence of unipotent orbits by Kleinbock and Margulis, and an extension by de la Salle of strong property (T) to representations of nonuniform lattices.
\end{abstract}





\section{Introduction}
\subsection{Statement of results}
The main result of this paper is the following:

\begin{custheorem}\label{main1} Let $\Gamma$ be a finite-index subgroup of $\Sl(m,\Z)$ and  let $M$ be a closed manifold of dimension $\dim(M) \leq m - 2$.  If  $\alpha\colon \Gamma \to \Diff^2(M)$ is a group homomorphism  then $\alpha(\Gamma)$ is finite\footnote{After this work was completed, Brown-Damjanovic-Zhang showed that some modifications of our arguments also give a proof for $C^1$ diffeomorphisms \cite{BDZ}.}.  In addition, if $\omega$ is a volume form on $M$, $m>2$ and if $\dim(M) \leq m-1$, then if and $\alpha\colon \Gamma \to \Diff^2(M, \omega)$ is a group homomorphism then $\alpha(\Gamma)$ is finite.
\end{custheorem}

For $m\ge 3$, we remark that the conclusion of Theorem \ref{main1} is known for actions on the circle by results of Witte Morris  \cite{MR1198459} (see also \cite{MR1703323,MR1911660} for actions by more general lattices on the circle) and for volume-preserving actions on surfaces by results of Franks and Handel and of Polterovich  \cite{MR2219247, MR1946555}.   The proof in this paper  requires that $m\ge 4$ though we expect it can be modified to cover actions by $\Sl(3,\Z)$; since these results are not new, we only present the case for $m\ge 4$.  While this is a very special case of Zimmer's conjecture, it is a key example.  For instance, the version of Zimmer's conjecture restated by Margulis in his problem list \cite{MR1754775} is a special case of Theorem \ref{main1}.

Note that if $\Gamma$ is a finite-index subgroup of  $\Sl(m,\Z)$ acting on compact manifold $M$, we may induce an action of $\Sl(m,\Z)$  on a (possibly non-connected) compact manifold $\td M= (\Sl(m,\R)\times M)/\sim$ where $(\gamma,x)\sim (\gamma', x')$ if there is $\hat \gamma\in \Gamma$ with $\gamma' = \gamma \hat \gamma$ and $x' = \alpha(\hat \gamma\inv)(x)$.
Connectedness of $M$ is neither assumed nor is it  used in either the proof of Theorem \ref{main1} or in \cite{BFH}.  Thus, for the remainder we will simply  assume $\Gamma = \Sl(m,\Z)$.

This paper is a first step in extending the results in \cite{BFH} to the case where $\Gamma$ is a nonuniform lattice in a split simple Lie group $G$ and the strategy of the proof of Theorem \ref{main1} relies strongly on the strategy used in \cite{BFH}.  In the remainder of the introduction, we recall the proof in the cocompact case, indicate where the difficulties arise in the nonuniform case, and outline the proof of Theorem \ref{main1}.  At the end of the introduction we make some remarks on other approaches and difficulties we encountered.  

 We recall a   key definition from \cite{BFH}.
Let $\Gamma$ be a finitely generated group.  Let $\ell\colon \Gamma \to \N$ denote the word-length function with respect to some choice of finite generating set for $\Gamma$.  Given a $C^1$ diffeomorphism $f\colon M\to M$ let $\|Df \| = \sup_{x \in M} \|D_xf \|$ (for some choice of norm on $TM$).
\begin{definition}
An action $\alpha\colon \Gamma\to \diff^1(M)$   has \emph{uniform subexponential growth of derivatives} if
\begin{equation}\label{eq:USEGOD}\text{   for every $\e>0$, there is $C_{\e}$ such  that  $\|D\alpha(\gamma)\| \leq C_{\e}e^{\e \ell(\gamma)}$ for all $\gamma\in \Gamma.$}\end{equation}
\end{definition}

The  main result of the paper is the following: 
 \begin{custheorem}\label{main2} For {$m\ge 4$}, let $\Gamma= \Sl(m,\Z)$ and let $M$ be a closed manifold.
 \begin{enumerate}
 \item If  $\dim(M) \leq m - 2$ then any action  $\alpha\colon  \Gamma \to \Diff^2(M)$ has {uniform subexponential growth of derivatives};
 \item  if $\omega$ is a volume form on $M$ and $\dim(M) \leq m-1$ then any action  $\alpha\colon \Gamma \to \Diff^2(M, \omega)$ has {uniform subexponential growth of derivatives}.
 \end{enumerate}
\end{custheorem}

To deduce Theorem \ref{main1} from  Theorem \ref{main2}, we apply \cite[Theorem 2.9]{BFH} and de la Salle's recent result establishing  strong property $(T)$ for nonuniform lattices  \cite[Theorem 1.2]{delaSallenonuniform} and  conclude that any action $\alpha$ as  in Theorem \ref{main1}  preserves a continuous Riemannian metric.  For clarity, we point out that we need de la Salle's Theorem $1.2$ and not his Theorem $1.1$ because we need the measures converging to the projection to be positive measures.  That Theorem \cite[Theorem 1.2]{delaSallenonuniform} provides positive measures where \cite[Theorem 1.1]{delaSallenonuniform} does not is further clarified in \cite[Section 2.3]{delaSallenonuniform}.   Once a continuous invariant metric is preserved, the   image of any homomorphism $\alpha$  in Theorem \ref{main1} is contained in  a compact Lie group $K$.  All such homomorphisms necessarily  have finite image due to the presence of unipotent elements in $\Sl(m,\Z)$.  We remark that  while the finiteness of the image of $\alpha$  was deduced using Margulis's superrigidity theorem in  \cite{BFH},  it is unnecessary in the setting  of Theorem \ref{main1} since, as  any unipotent element of $\Sl(m,\Z)$ lies in the center of some integral Heisenberg subgroup of $\Sl(m,\Z)$, all unipotent elements have finite image in $K$ and therefore so does $\Sl(m,\Z)$. 

\subsection{Review of the cocompact case} To explain the proof of Theorem \ref{main2}, we briefly  explain the   difficulties in extending the arguments from \cite{BFH} to the setting of actions by nonuniform lattices.  We begin by recalling the proof in the cocompact setting.


In both  \cite{BFH} and the proof of Theorem \ref{main2}, we consider a fiber bundle $$M \rightarrow M^{\alpha}:=(G \times M)/\Gamma \xrightarrow{\pi} G/\Gamma$$ which allows us to replace  the $\Gamma$-action on $M$ with a $G$-action on $M^\alpha$.  In the case that $\Gamma$ is cocompact,  showing subexponential growth of derivatives of the $\Gamma$-action is equivalent to showing subexponential growth of the fiberwise derivative cocycle for the $G$-action.

%

To prove such subexponential growth for the $G$-action on $M^{\alpha}$ we argued by contradiction to obtain a sequence of points $x_n \in M^{\alpha}$ and semisimple elements $a_n$ in a Cartan subgroup $A \subset G$  which satisfy $\|\restrict{D_{x_n}a_n}{F}\|  \geq e^{\lambda d(a_n, \Id)}$ for some $\lambda>0$. Here $D_{x} g$ denotes the derivative of translation by $g$ at $x\in M^\alpha$, $F$ is the fiberwise tangent bundle of $M^\alpha$, and $\restrict{D_{x_n}a_n}{F}$ is the restriction of $ D_{x_n}a_n $ to ${F(x_n)}$.


The pairs $(x_n, a_n)$ determine empirical measures $\mu_n$ on $M^{\alpha}$ supported on the orbit $\{a_n^s(x_n): 0\le s\le t_n\}$
which accumulate on a  
measure $\mu$ that is $a$-invariant for some $a \in A$ and has a positive Lyapunov exponent for the fiberwise derivative cocycle of size at least  $\lambda$.
Using classical results in homogeneous dynamics in conjunction with the key proposition from \cite{AWBFRHZW-latticemeasure}, we averaged the measure $\mu$ to obtain a $G$-invariant measure $\mu'$ on $M^\alpha$ with a non-zero fiberwise Lyapunov exponent;  the existence of such a measure $\mu'$ contradicts Zimmer's cocycle superrigidity theorem.  

\subsection{Difficulties in the nonuniform setting.} When $\Gamma$ is nonuniform
the space $M^{\alpha}$ is not compact and the sequence of empirical measures $\mu_n$ might  diverge to infinity in $M^{\alpha}$; that is, in the limit we might have a ``loss of mass".
Additionally, even if the measures $\{\mu_n\}$ satisfy some tightness criteria so as to prevent escape of mass, one might have  ``escape of Lyapunov exponents:" for a limiting measure $\mu$, the Lyapunov exponents may be infinite or the value could drop below the value expected by the growth of fiberwise cocycles along the orbits $\{a^s(x_n): 0\le s\le t_n\}$.
For instance,  the contribution to the exponential growth of derivatives along the sequence of empirical measures could arise primarily from excursions of orbits deep into the cusp.  If one makes na\"ive computations with the \emph{return cocycle} $\beta\colon G \times G/\Gamma \rightarrow \Gamma$ (measuring for $x$ in a fundamental domain $D$ the element of $\Gamma$ needed to bring $gx$ back to a $D$)  one in fact expects that the fiberwise derivative are very large for translations of points far out in the cusp since the orbits of such points cross a large number of fundamental domains. 
The weakest consequence of this observation is that subexponential growth of the fiberwise derivative of the induced $G$-action is much stronger than subexponential growth of derivatives of the $\Gamma$-action.  While we   still     work with the induced $G$-action and the fiberwise derivative in many places, the arguments become more complicated than in the cocompact case.

In the  homogeneous dynamics literature, there  are many tools to study  escape of mass.  Controlling the escape of Lyapunov exponents seems to be more novel.  To rule out escape of mass, it suffices to prove  tightness of family of  measures $\{\mu_n\}.$
To control Lyapunov exponents, we introduce a quantitative tightness condition: we construct measures  $\{\mu_n\}$ with \emph{uniformly exponentially small mass in the cusps}. See Section \ref{sec:sampson}.
 It is a standard computation to show the Haar measure on $\Sl(m,\R)/\Sl(m,\Z)$  (or any  $G/\Gamma$ where $G$ is semisimple and $\Gamma$ is a lattice) has exponentially small mass in the cusps.  

\subsection{Outline of proof} 
 With the above difficulties in mind, we outline   the strategy of the proof of Theorem \ref{main2}.  Lubotzky, Mozes and Raghunathan proved  that $\Sl(m, \Z)$ is quasi-isometrically embedded in $\Sl(m, \R)$.  And in this special case, they give a proof that every element  $\gamma\in \Sl(m, \Z)$ can be written   as a product of at most $m^2$ elements $\delta_i$ contained in  canonical copies of $\Sl(2, \Z)$ determined by pairs of standard basis vectors for $\R^m$; moreover the word-length of each $\delta_i$ is at most proportional to the word-length of $\gamma$ \cite[Corollary 3]{MR1244421}.  (We note however that such effective generation of $\Gamma$ only holds for $\Sl(m,\Z)$;  for  the general case, in  \cite{MR1828742} a weaker generation of $\Gamma$ in terms of $\Q$-rank 1 subgroups is shown.) Thus, to show uniform subexponential growth of derivatives for the action of $\Sl(m, \Z)$, it suffices to show uniform subexponential growth of derivatives for the restriction of our action to each canonical copy of $\Sl(2,\Z)$.


We first obtain uniform subexponential growth of derivatives for the unipotent elements in $\Sl(2,\Z)$  in Section \ref{unipotents}.  See Proposition \ref{unipotentisgood}.
The strategy is to consider a subgroup of the form $\Sl(2,\Z)\ltimes \Z^2 \subset \Sl(m,\Z)$. We first prove that a large proportion of    elements in $\Sl(2, \Z)$ satisfy \eqref{eq:USEGOD}. To prove this, we use that if $a^t:= \text{diag}(e^t, e^{-t}) $ 
 then a typical $a^t$-orbit in $\SL(2, \R)/\Sl(2, \Z)$ equidistributes to the Haar measure.  
 In particular, for the empirical measures along such $a$-orbits we   apply the techniques from \cite{BFH} to show subexponential growth of fiberwise derivatives along such orbits and conclude that a large proportion of $\Sl(2, \Z)$ satisfies \eqref{eq:USEGOD}.   See Proposition \ref{mainunipo}. The proof of this fact repeats most of the ideas and techniques from \cite{BFH} as well a quantitative non-divergence of unipotent averages following  Kleinbock and Margulis. 
The precise averaging procedure is different here than in \cite{BFH}.

Having shown Proposition \ref{mainunipo}, we consider the $\Sl(2,\Z)$-action on the normal  subgroup $\Z^2$ of $\Sl(2,\Z)\ltimes \Z^2$  to show that   for every $n \geq 0$, the ball $B_n$ of radius $n$ in $\Z^2$ contains a positive-density subset  of unipotent elements satisfying \eqref{eq:USEGOD}.  Taking iterated sumsets of such good unipotent elements of $B_n(\Z^2)$  with   a finite set  one obtains uniform subexponential growth of derivatives for every element in $B_n$. This relies heavily on the fact that $\Z^2$ is abelian. See Subsection \ref{sec:mutualmastication}.


It is worth noting that the subgroups of the form $\Sl(2,\Z)\ltimes \Z^2\subset \Gamma$ are also considered in the work of Lubotzky, Mozes, and Raghunathan in \cite{MR1244421} as well as in Margulis's early constructions of expander graphs and   subsequent work on property (T) and expanders \cite{MR0484767}. 

Having established Proposition \ref{unipotentisgood},  we assume for the sake of  contradiction that the restriction of $\alpha$ to $\Sl(2,\Z)$ fails to exhibit uniform subexponential growth of derivatives.  We  obtain in Subsection \ref{maximal} a sequence $\zeta_n$ of $a^t$-orbit segments  in $\Sl(2, \R)/\Sl(2, \Z)$ which drift  only a sub-linear distance into the cusp with respect to their length and accumulate exponential growth of the fiberwise derivative.
Here we use that orbits deep in the cusp of $\Sl(2,\R)/\Sl(2,\Z)$ correspond to unipotent deck transformations and  that Proposition \ref{unipotentisgood} implies that these do not contribute to the exponential growth of the fiberwise derivative.
Here, we heavily use the structure of 
$\Sl(2,\Z)$ subgroups. 

We   promote the family of  orbit segments $\zeta_n$  in $M^\alpha$ to a family of  measures $\{\mu_n\}$ all of whose  subsequential limits   are $A$-invariant measures $\mu$ on $M^{\alpha}$ with non-zero fiberwise exponents. To construct $\mu_n$,   we construct a  \Folner sequence $F_n \subset G$ inside a  solvable subgroup $AN'$ where $A$ is the full Cartan subgroup of $\Sl(m,\R)$ and $N'$ is a well-chosen abelian subgroup of unipotent elements. We average our orbit segments $\zeta_n$  over  $F_n$ to obtain  the sequence of measures $\mu_n$ in $M^{\alpha}$. In general, \Folner sets for $AN'$ are subsets which are linearly large in the $A$-direction and exponentially large in the $N'$ direction. In our case the $N'$-part will not affect the Lyapunov exponent because we work inside a subset where the return cocycle $\beta$ restricted to $N'$ takes unipotent values and we have already proven subexponential growth of the fiberwise derivatives for unipotent elements.

The fact that $\mu_n$ behaves well in the cusp is due to two facts: First, the segments obtained in Subsection \ref{maximal} do not drift  too deep into the cusp of $\Sl(2,\R)/\Sl(2,\Z)$.   Second, we choose our subgroup $N'$ such that the $N'$-orbits of each  point  along each $\zeta_n$ is a closed torus that is  well-behaved when translated by $A$.  The argument here is related to the fact  closed horocycles in the cusp of $\Sl(2,\R)/\Sl(2,\Z)$ equidistribute  to the Haar measure when flowed backwards by the geodesic flow.

To finish the argument, we show that any $AN'$-invariant measure on $M^{\alpha}$ projects to Haar measure on $\Sl(m,\R)/\Sl(m,\Z)$ using  Ratner's measure classification and equidistribution theorems.  Then, as in \cite{BFH}, we can use \cite[Proposition 5.1]{AWBFRHZW-latticemeasure} and argue as in the cocompact case in \cite{BFH} show that  $\mu$ is in fact $G$-invariant and thereby obtain a contradiction with Zimmer's cocycle superrigidity theorem.

\subsection{A few remarks on other approaches.}
 \label{subsection:failure}

 We close the introduction by making some remarks on other approaches, particularly other approaches for controlling the  escape of mass.  We emphasize here that one key difficulty for all approaches is that we are not able to control the ``images" of the cocycle $\beta\colon G \times G/\Gamma \rightarrow \Gamma$ in either our special case or in general. To understand this remark better, consider first the case where $G=\Sl(2,\R)$ and $\Gamma =\Sl(2,\Z)$. If we take a one-parameter subgroup $c(t)<\Sl(2,\R)$ and take the trajectory $c(t)x$ for $t$ in some interval $[0,T]$ and assume and assume the entire trajectory on $G/\Gamma$ lies deep enough in the cusp, then $\beta(a(t),x)$ is necessarily unipotent for all $t$ in $[0,T]$.  No similar statement is true for $G=\Sl(m,\R)$ and $\Gamma=\Sl(m,\Z)$.  In fact analogous statements are true if and only if $\Gamma$ has $\Q$-rank one, this is closely related to the fact that higher $\Q$-rank locally symmetric spaces are $1$-connected at infinity.  This forces us to ``factor" the action into actions of rank-one subgroups in order to  control the growth of derivatives.


One might hope to obtain subexponential growth of derivatives more directly for all elements of $\Sl(2,\Z)$, or even directly in $\Sl(m,\Z)$, by proving better estimates on the size of the ``generic" subsets of $\Sl(2,\R)$  (or $\Sl(m, \R)$) whose $A$-orbits define empirical measures satisfying some tightness condition. While one can get good estimates on the size of the sets in Proposition \ref{mainunipo} using Margulis functions and large deviation estimates as in  \cite{MR2247652, MR2787598}, the resulting estimates are not sharp enough to allow us to prove subexponential growth of derivatives. One can compare with the conjectures in \cite{KKLM} about loss of mass.

An elementary related question is the following:  Let $B_n$ be a ball of radius $n$ in a Lie group $G$ (or a lattice $\Gamma$) and suppose there exists subset $S_n$ of $B_n$ such that $S_n$ and $B_n$ have more or less equal mass, meaning that: $$\frac{vol(B_n \setminus S_n)}{vol(B_n)} < \e_n $$ for a certain sequence $\e_n$ of numbers converging to zero. Does there exists an integer $k$ (independent of $n$) such that for $n$ large:

\begin{equation}\label{mof}
B_{n} \subset S_n*S_n*\stackrel{k}{\cdots}*S_n
\end{equation}
Observe that the question depends on how fast $\e_n$ is decreasing and on the group $G$. For example if $G$ abelian, $\e_n$ can be a sufficiently small constant as a consequence of Proposition \ref{babycombinatorics}.  Also, it is not hard to see that for any group $G$ the existence of $k$ is guaranteed if $\e_n$ decreases exponentially quickly.  So the real question is how fast $\e_n$ has to decrease to zero in order for this statement to hold.  Does (\ref{mof}) holds for $ G = \Sl_3(\Z)$ and $\e_n = 2^{-n^c}$  for some $c < 1$?  If the answer to this question is yes, then it would be possible to approach our results via Margulis functions and large deviation estimates.


\subsection*{Acknowledgements}  We thank Dave Witte Morris for his generous willingness to answer questions of all sorts throughout the production of this paper and \cite{BFH}.  We also thank to Shirali Kadyrov, Jayadev Athreya and Alex Eskin for helpful conversations, particularly on the material in Subsection \ref{subsection:failure} and Mikael de la Salle for many helpful conversations regarding strong property $(T)$.  {\blue We also thank the anonymous referee for a very careful reading and numerous comments which helped improve the exposition.}

\section{Standing notation} 
We review the notation introduced in \cite{BFH} and establish some standing notation and conventions as well as state some facts  used in the remainder of the paper.

\subsection{Lie theoretic and geometric notation}
We write $G= \Sl(m,\R)$ and $\Gamma= \Sl(m,\Z)$.  Let $\lieg$ denote the Lie algebra of $G$.  Let $\Id$ denote the identity element of  $G$.
We fix the standard Cartan involution $\theta \colon \lieg\to \lieg$ given by $\theta (X) = -X^{t}$ and write $\liek$ and $\liep$, respectively,  for the $+1$ and $-1$ eigenspaces of $\theta$.   Define $\liea$  to be a maximal abelian subalgebra of $\liep$.  Then $\liea$ is the vector space of diagonal matrices.

The roots of $\lieg$ are the  linear functionals $\beta_{i,j}\in \liea^*$ defined as $$\beta_{i,j}(\diag (t_1, \dots, t_m)) = t_i-t_j.$$  The simple positive roots are $\alpha_{j} = \beta_{j,j+1}$ and the positive roots are the positive integral combinations of $\{\alpha_j\}$ that are still roots.

For a root $\beta$, write $\lieg^\beta$ for the associated root space.   Each root space $\lieg^\beta$ exponentiates to a 1-parameter unipotent subgroup $U^\beta\subset G$.
The Lie subalgebra $\lien$ generated by all root spaces $\lieg^\beta$ for positive roots $\beta$, coincides with the Lie algebra of all strictly upper-triangular matrices.

Let $A,N,$ and $ K$ be the {analytic subgroups} of $G$ corresponding to $\liea, \lien$ and $\liek$.  Then
\begin{enumerate}
\item $A= \exp (\liea)$ is the group of all diagonal matrices with positive entries.  $A$ is an abelian group and we identity linear functionals on $\liea$ with linear functionals on $A$ via the exponential map $\exp\colon \liea\to A$;
\item $N= \exp (\lien)$ is the group of upper-triangular matrices with $1$s on the diagonal;
\item $K = \So(m)$.
\end{enumerate}

The Weyl group of $G$ is the group of permutation matrices.  This acts transitively on the set of all roots $\Sigma$.

For $1\le i,j\le m$, the subgroup of $G$ generated by $U^{\beta_{i,j}}$ and $U^{\beta_{j,i}}$ is isomorphic to $\Sl(2,\R)$.  We denote this subgroup by  $H_{i,j} = \Sl_{{i,j}}(2,\R).$    Then $\Lambda_{i,j}:=H_{i,j} \cap \Gamma$ is a lattice in $\Sl_{{i,j}}(2,\R)$   isomorphic to $\Sl(2,\Z)$.  
 Note then that $X_{i,j}:= H_{i,j}/\Lambda_{i,j}$ is the unit tangent bundle to the modular surface.
We will use the standard notation $E_{i,j}$ for an elementary matrix with 1s on the diagonal and in the $(i,j)$-place and 0s everywhere else.  Note that $E_{i,j}$ and $E_{j,i}$ generate $\Lambda_{i,j}$.

We equip $G$  with a  left-$K$-invariant and  right-$G$-invariant metric.  Such a metric is unique up to scaling.  Let $d$ denote be the induced distance on $G$.
With respect to this metric and distance $d$, each $H_{i,j}$ is geodesically embedded.  By rescaling the metric, 
  we may assume the restriction of $d$ to each $H_{i,j}$ {\blue coincides with the standard metric of constant curvature $-1$ on the upper half plane $\So(2)\backslash \Sl(2,\R)$.}
This metric has the following properties that we exploit throughout.
\begin{enumerate}
\item  For any matrix norm $\|\cdot \|$ on $H_{i,j}\simeq \Sl(2,\R)$  there is a $C_1$ such that
\begin{equation}\label{normdistance}  {2}\log \|A\| - C_1 \leq d(A, \Id) \leq    {2}\log \|A\| + C_1\end{equation}
 for all $A\in H_{i,j}$.
 \item   Let $B(\Id,r)$ denote the metric ball of radius $r$  in   $H_{i,j}$ centered at $\id$.  Then with respect to the induced Riemannian  volume on $H_{i,j}$ we have $$\mathrm{vol}(B(\id,r))= 4\pi (\cosh(r) -1)\le 4\pi e^r$$
and for all sufficiently large $r>0$
\begin{equation}\label{decapitationOrImpeachment?}
\mathrm{vol}(B(x,r))\ge e^r.
\end{equation}
\item For any matrix norm $\|\cdot \|$ on $\Sl(m,\R)$, there are constants $C_0>1$ and $\kappa>1$ such that for any matrix $A\in \Sl(m,\R)$ we have 
    \begin{equation}\label{eq:easy} \begin{gathered}
\kappa\inv\log \|A\| -  C_0
 \le d(A,\id) \le \kappa \log \|A\| +  C_0.
%
%
\end{gathered}
\end{equation}
\item In particular, there are $C_2$ and $C_3$ so that if $E_{i,j}\in \Sl(m,\Z)$ is an elementary unipotent matrix  then
\begin{equation}\label{unipotentgrowth}d(E_{i,j}^k,\id) \le C_2 \log  k + C_3.\end{equation}  
\end{enumerate}

\subsection{Suspension space and induced $G$-action} \label{sec:lolo}

 Let $M^{\alpha} = (G \times M)/\Gamma$ be the fiber-bundle over $\Sl(m,\R)/\Sl(m, \Z)$ obtained as follows: on $G\times M$ let $\Gamma$ act as
$$(g,x)\cdot \gamma = (g\gamma, \alpha(\gamma\inv)(x))$$
and let $G$ act as $$g'\cdot(g,x) = (g'g,x).$$
The $G$-action on $G\times M$ descends to a $G$-action  on the quotient $M^{\alpha} = (G \times M)/\Gamma$.
 Let $\pi\colon M^\alpha \to \Sl(m,\R)/\Sl(m, \Z)$ be the canonical projection.
As in \cite{BFH}, we   write  $F= \ker D\pi$ for the fiberwise tangent bundle to $M^{\alpha}$.  Write  $\P F$ for the projectivization of the fiberwise tangent bundle.
We write $\restrict { D_x g} F \colon F(x) \rightarrow F(gx)$ for the fiberwise derivative as in \cite{BFH}.  For $(x,[v])\in \P F$ and $g\in G$, write $$g\cdot (x,[v]) := (g\cdot x, [\restrict { D_x g} {F (x)} v])$$ for the action of $g$ on $\P F$ induced by $\restrict { D_x g} F$.

We follow  \cite[Section  2.1]{AWBFRHZW-latticemeasure} and equip $G\times M$ with a $C^1$ Riemannian metric $\langle \cdot, \cdot \rangle$ with the following properties:
\begin{enumerate}
\item $\langle \cdot, \cdot \rangle$ is $\Gamma$-invariant.
\item for $x\in M$ and $g\in G$, under the canonical identification of the $G$-orbit of $(g,x)$ with $G$, the restriction of $\langle \cdot, \cdot \rangle$ to the   $G$-orbit of $(g,x)$  coincides with the fixed right-invariant metric on $G$.
\item There is a Siegel fundamental set $D\subset G$ and $C>1$ such that for any $g_1,g_2\in D$, the map $(g_1,x)\mapsto (g_2,x)$ distorts the restrictions of $\langle \cdot, \cdot \rangle$ to $\{g_1\} \times M$ and $\{g_2\} \times M$  by at most $C$.
\end{enumerate}
The metric then descends to a $C^1$ Riemannian metric on $M^\alpha$.  {\blue Note that by averaging the metric over the left action of $K$, we may also assume that the metric on $M^\alpha$ is left-$K$-invariant.   This, in particular, implies the right-invariant metric on $G$ in $(2)$ above is chosen to be left-$K$-invariant.} 

To analyze the coarse dynamics of the suspension action, it is often useful to consider the \emph{return cocycle}
$\beta\colon G \times G/\Gamma \rightarrow \Gamma$.  This cocycle is defined relative to a fundamental
domain $\calF$ for the right $\Gamma$-action on $G$.  For any $x\in G/\Gamma$, take $\tilde x$ to be the unique
lift of $x$ in $\calF$ and define $\beta(g,x)$ to be the unique element of $\gamma\in \Gamma$ such that
$g\tilde{x}\gamma^{-1} \in \mathcal F$.  Any two choices of fundamental domain for $\Gamma$
define cohomologous cocycles but we require a choice of well-controlled fundamental domains $\calF$.  Namely, we choose $\calF$ to either be contained in a Siegel fundamental set or to be a Dirichlet domain for the identity.  With these choices, we have the following.

{\blue Let $\wtd {\mathcal D} \subset \Sl(m,\R)$ denote the Dirichlet domain of the identity for the $\Sl(m,\Z)$ action on $\Sl(m,\R)$; that is
$$\wtd {\mathcal D}:= \{ g\in \Sl(n,\R): d(g,\id) \le d(g\gamma, \id) \text{ for all $\gamma\in \Sl(m,\Z)$}\}.$$
Since each $H_{i,j}$ is geodesically embedded in $\Sl(m,\R)$ and since $\Lambda_{i,j} = H_{i,j} \cap \Sl(m,Z)$, it follows
\begin{equation} \label{eq:dirc}\mathcal D := H_{1,2} \cap  \wtd {\mathcal D}\end{equation}
is a Dirichlet domain of the identity for the $\Lambda_{1,2}$-action on $H_{1,2}$.  
Viewing $H_{1,2}\simeq \Sl(2,\R)$ acting  on the upper half-plane model of hyperbolic space $\mathbb H^2= \So(2)\backslash \Sl(2,\R)$ by M\"obius transformations ${\So(2)\backslash \mathcal D}$ is the standard Dirichlet domain for the modular surface, the hyperbolic triangle with endpoints at $1/2 + i\sqrt 3/2$, $-1/2 + i\sqrt 3/2$, and $\infty$.
}

\begin{lemma}
\label{lemma:fromlmr}
If $\mathcal F$ is either contained in either a Siegel fundamental set or a Dirichlet domain for the identity then
there is a constant $C$ such that  for all $g\in G$ and $x\in G/\Gamma$ $$\ell (\beta(g,x)) \leq C d(g,e)+ C d(x, \Gamma) + C.$$
\end{lemma}
In the above lemma, $\ell$ is the word-length of $\beta(g,x)$, $d(g,e)$ is the distance from $g$ to $e$ in $G$, and $d(x, \Gamma)$ is the distance from $x\in G/\Gamma$ to the identity coset $\Gamma$ in $G/\Gamma$.
For a Dirichlet domain for the identity, the Lemma is shown in \cite[\S 2]{MR1767270}; for fundamental domains contained in Siegel fundamental sets, the estimate follows from  \cite[Corollary 3.19]{MR2039990} and   the fact that the distance to the identity in a  Siegel domain is quasi-Lipschitz equivalent to the distance to the identity in the quotient   $G/\Gamma$.   Both estimates heavily use
the main theorem of Lubotzky, Mozes, and Raghunathan \cite{MR1244421,MR1828742} to compare the word-length of $\beta(g,x)\in \Sl(m,\Z)$ with   $\log(\|\beta(g,x)\|)$. 

{\blue  Fix once and for all a fundamental domain $\mathcal F\subset \wtd {\mathcal D} \subset \Sl(m,\R)$.}

The estimates in  Lemma \ref{lemma:fromlmr} is often used to obtain integrability properties of $\beta$ and   related cocycles with respect  to the Haar measure on $G/\Gamma$.
As the function $x\mapsto d(x, \Gamma)$ is in  $L^p(G/\Gamma,\mathrm{Haar})$ for any compact set $K\subset G$ we have that $$x\mapsto \sup _{g\in K} \ell(\beta (g,x))$$ is in $L^p(G/\Gamma,\text{Haar})$ for all $p\ge 1$.
In the sequel, we typically do not directly use the integrability properties  (since we work with measures other than Haar)  but rather  the   estimate in Lemma \ref{lemma:fromlmr}.




\section{Preliminaries on measures, averaging, and Lyapunov exponents}
\label{section:preliminaries}

We present a number of technical facts regarding invariant measures, equidistribution, averaging, and Lyapunov exponents  that will be used in the remainder of the paper.

\subsection{Ratner's measure classification and equidistribution theorems}
We recall Ratner's theorems on equidistribution of unipotent flows.
Let $U = \{u(t) = \exp_\lieg (t X)\}$ be a 1-parameter unipotent subgroup in $G$.
Given any Borel probability measure $\mu$ on $G/\Gamma$ let
$$U^T\ast \mu := \frac 1 T \int_0^T  u(t)_* \mu \ d t.$$

\begin{theorem}[Ratner]\label{thm:ratner}
Let $U = \{u(t) = \exp_\lieg (t X)\}$ be a 1-parameter unipotent subgroup and consider the action on $G/\Gamma$.  The following hold:
\begin{enumlemma}
	\item \label{ratner1}Every ergodic, $U$-invariant probability measure on $G/\Gamma$ is homogeneous \cite[Theorem 1]{MR1262705}.
	\item \label{ratner2} The orbit closure $\orb_x:= \overline{\{u \cdot x :u\in U\} }$ is homogeneous for every $x\in G/\Gamma$  \cite[Theorem 3]{MR1262705}.
	\item \label{ratner4} The orbit  ${U \cdot x }$ equidistributes in $\orb_x$; that is $U^T\ast \delta_x$ converges to the Haar  measure on $\orb_x$ as $T\to \infty$.

\item \label{ratner3} \label{thisone} Let $\beta$ be a root of $\lieg$ and let $\sl_\beta(2)\subset \lieg$ be the Lie subalgebra generated by $\lieg^\beta$ and $\lieg^{-\beta}$.  Let $e,f,h\subset \sl_\beta(2)$ be an $\sl(2,\R)$ triple with $e\in \lieg^\beta$ and $f\in \lieg^{-\beta}$ and let $\lieh^\beta = \mathrm{span} (h)$.  Let $H^\beta= \exp \lieh^\beta$.

Let $\mu$ be a   $U^{\beta}$-invariant Borel probability measure on $G/\Gamma$.
	If $\mu$ is $H^\beta$-invariant, then $\mu$ is $U^{-\beta}$-invariant.
\end{enumlemma}
\end{theorem}
Conclusion \ref{thisone} follows from \cite[Proposition 2.1]{MR1135878} and the structure of $\mathfrak{sl}(2,\R)$-triples. See also the discussion in the paragraph preceding  \cite[Theorem 9]{MR1262705}.  In our earlier work on cocompact lattices \cite{BFH}, we averaged over higher-dimensional unipotent subgroups and required
a variant of (c) due to Nimish Shah  \cite{MR1291701}.  Here we only average over one-dimensional root subgroups and can use the earlier version due to  Ratner.

From Theorem  \ref{thm:ratner}, for any probability measure $\mu$ on $G/\Gamma$ it follows
that   the weak-$*$ limit $$U\ast \mu:= \lim_{T\to \infty} U^T\ast \mu $$ exists and that the $U$-ergodic components of $U\ast \mu$ are  homogeneous.

\subsection{Measures with exponentially small mass in the cusps}\label{sec:sampson}
We now define precisely  the notion of measures with exponentially small mass in the cusps from the introduction.
Let $(X,d)$ be a complete,  second countable, metric space.  Then $X$ is Polish.  Let  $\mu$ be a finite Borel (and hence Radon) measure on $X$.  We say that $\mu$ has \emph{exponentially small mass in the cusps with exponent $\eta_\mu$} if for all  $0<\eta<\eta_\mu$
\begin{equation}\label{eq:racistPrez}\int_X e^{\eta d(x_0, x)} \ d \mu(x) <\infty\end{equation}
for some (and hence any) choice of base point $x_0\in X$.
We say that a collection $\mathcal M=\{\mu_\zeta \}$ of probability measures on $X$ has \emph{uniformly exponentially small mass in the cusps with exponent $\eta_0$} if for all  $0<\eta<\eta_0$
$$\sup_{\mu_\zeta\in \mathcal M} \left\{\int e^{\eta d(x_0, x)} \ d \mu_\zeta(x)\right\} <\infty.$$

Below, we often work in in the setting $X= G/\Gamma$ where $G = \Sl(m,\R)$ and $\Gamma= \Sl(m,\Z)$ and where $d$ the distance induced from a right-invariant metric on $G$.
When $X=   \Sl(m,\R)/\Sl(m,\Z)$ we interpret  a point $x= g\Gamma \in G/\Gamma$ as a unimodular lattice $\Lambda_g = g\cdot \Z^m$.   Fix any norm on $\R^m$ and define the systole of a lattice  $\Lambda\subset \R^m$ to be $$\delta(\Lambda):= \inf\left\{ \| v\| : v\in \Lambda\sm \{0\}\right\}.$$
We have that  \begin{equation}\label{eq:ploy}c_1\le \frac{1- \log(\delta(\Lambda_g))}{1+ (d(g\Gamma, e\Gamma))} \le c_2\end{equation}
for some constants whence $$C_1 e^{c_1d(g\Gamma, e\Gamma)} \le \frac {1}{\delta(\Lambda_g)} \le C_2 e^{c_2 d(g\Gamma, e\Gamma)}.$$
Thus, if we only care about finding a positive exponent $\eta_\mu>0$  such that \eqref{eq:racistPrez} holds for all $\eta<\eta_\mu$, it suffices to find $\eta$ such that
\begin{equation}\label{eq:systole} \int \delta(\Lambda_g)^{-\eta}  \ d \mu(g\Gamma) <\infty.\end{equation}
We define the \emph{systolic exponent} $\eta^S_\mu$ to be the supremum of all $\eta$ satisfying \eqref{eq:systole}.


In the sequel, we will frequently use the following proposition to avoid  escape of mass into the cusps of $G/\Gamma$ when  averaging a measure along a unipotent flow.
\begin{proposition}\label{prop:bananas}Let $U$ be a 1-parameter unipotent subgroup of $G$.
Let $\mu$ be a probability measure on $X = \Sl(m,\R)/\Sl(m,\Z)$ with  exponentially small mass  in the cusps.  Then the family of  measures $$\{U^T\ast \mu: T\in \R\} \cup \{U\ast \mu\}$$
has uniformly exponentially small mass in the cusps.   
\end{proposition}

\subsection{Proof of Proposition \ref{prop:bananas}}

We first show that the family of averaged measures  $$\{U^T\ast \mu: T\in \R\}$$ has uniformly exponentially small mass in the cusps.  The key idea is to use the quantitative non-divergence of unipotent orbits following Kleinbock and Margulis.
\begin{lemma}
\label{lemma:translates}
Let $\mu$ be a probability measure on $X = \Sl(m,\R)/\Sl(m,\Z)$ with    exponentially small mass in the cusps   and systolic exponent $\eta^S_\mu$.

Then the  family of measures  $\{U^T\ast \mu: T\in \R\}$ has uniformly exponentially small mass in the cusps with   systolic exponent   $\min\{\eta^S_\mu,  \frac 1{m^2}\}$.
\end{lemma}

\begin{proof}
Let $\Delta\subset \R^m$ be a discrete subgroup.  Let $\|\Delta\|$ denote the volume of $\Delta_\R /\Delta$ where  $\Delta_\R$ denotes the $\R$-span of $\Delta$.  It follows from Minkowski's lemma that there is a constant $c_m$ (depending only on $m$) such that if
$$\|\Delta\| \le (\rho')^{\mathrm{rk}(\Delta)}$$ then there is a non-zero vector $v\in \Delta$ with $\|v\| \le c_m \rho'$.
In particular, if $\delta(\Lambda)\ge \rho$ then for some constant $c'_m$ we have $$
\|\Delta\| \ge (c_m' \rho)^{\mathrm{rk}(\Delta)}$$ for all discrete subgroups $\Delta\subset \Lambda$.

From \cite[Theorem 5.3]{MR1652916} as extended in \cite[Theorem 0.1]{MR2434296}, there is a $C>1$ such that  for every $\Lambda_g\in G/\Gamma$ and $\epsilon > 0$, if $\delta(\Lambda_g)\ge \rho$  then, since $
\|\Delta\| \ge (c_n' \rho)^{\mathrm{rk}(\Delta)}$ for every discrete subgroup $\Delta\subset \Lambda_g$,
we have
\begin{equation}\label{KM}m\{ t\in [0,T] : \delta( \Lambda_{u_tg}) \le \epsilon \}  \le C\left(\frac {\epsilon}{(c_n')\inv \rho}\right)^{\frac 1 {m^2}}T =
\hat C\left(\frac {\epsilon}{ \rho}\right)^{\frac 1 {m^2}}T\end{equation}
where $m(A)$ is  the Lebesgue measure of the set $A\subset \R$.
Note that   \eqref{KM} still holds 
even in the case $\epsilon \ge \rho$.
Note that if $\beta<{\frac 1 {m^2}}$ then for $\epsilon <\rho$ we have
$$\left(\frac {\epsilon}{ \rho}\right)^{\frac 1 {m^2}}T<\left(\frac {\epsilon}{ \rho}\right)^{\beta}T.$$
In particular, when $\beta<\frac 1 {m^2}$ we   have (for all $\epsilon>0$ including   $\epsilon>\delta(\Lambda_g)$) that
$$m\{ t\in [0,T] : \delta( \Lambda_{u_tg}) \le \epsilon \}  \le \hat C\left(\frac {\epsilon}{\delta(\Lambda_g)}\right)^{\beta}T.$$

Then for $\eta>0$  and $\beta< \frac 1{m^2}$ we have
\begin{align*}
\int  [\delta(\Lambda_{g})]^{-\eta} \ d U^T\ast \mu (g) &=
\int _M \frac{1}{T} \int_{0}^{T} [\delta(\Lambda_{u_t g})]^{-\eta} \ d  t \ d  \mu(g) \\
&= \int _M \frac{1}{T}  \int_{0}^{\infty} m\{t\in [0,T] : [\delta(\Lambda_{u_t g})]^{-\eta} \ge \ell \}   \ d \ell  \ d  \mu(g) \\
&\le
\int _M \frac{1}{T}\left[T+  \int_{1}^{\infty} m\{t\in [0,T] : [\delta(\Lambda_{u_t g})]^{-\eta} \ge \ell \}   \ d \ell\right] \ d  \mu(g) \\
&= 1+
\int _M \frac{1}{T} \int_{1}^{\infty} m\{t\in [0,T] : [\delta(\Lambda_{u_t g})] \le {\ell^{-\frac 1\eta}} \} |  \ d \ell \ d  \mu(g) \\
& \le  1 + \int _M \frac{1}{T}  \int_1^\infty
\left[\hat C\left(\frac {1}{\ell ^{\frac 1\eta} \delta(\Lambda_g)}\right)^{\beta}T\right] \ d \ell
\ d  \mu(g) \\
& =1+  \hat C \left(\int _M
\left(\frac {1}{\delta(\Lambda_g)}\right)^{\beta} \ d  \mu(g)\right)
\left(\int _1 ^\infty \left(\frac {1}{\ell ^{\frac 1\eta}}\right)^{\beta}  \ d \ell
        \right)
\end{align*}
which is uniformly bounded in $T$ as long as $\eta<\beta< \min\{\eta_\mu^S,  \frac 1{m^2}\} $.
\end{proof}

For the limit measure $U \ast \mu= \lim_{T\to \infty} U^T\ast \mu$ we have the following which holds in full generality.
\begin{lemma}\label{lem:pastrydough}

Let $(X,d)$ be a complete,  second countable, metric space.
Let $\nu_j$ be a sequence of  Borel probability  measures on $X$ converging in the weak-$*$ topology to a measure $\nu$.   If the family
$\{ \nu_j\}$ has uniformly exponentially small mass in the cusps with exponent   $\eta_0$ then the limit $\nu$ has   exponentially small mass in the cusps with  exponent   $\eta_0$.   \end{lemma}
\begin{proof}
We have that $\nu_j\to \nu$ in the weak-$*$ topology.  In particular, for any closed set $C\subset X$ and open set  $U\subset X$ we have
$$\limsup_{j\to \infty} \nu_j (C) \le  \nu(C) \quad \text{and} \quad \liminf_{j\to \infty} \nu_j (U) \ge  \nu(U).$$

Fix $0<\eta'<\eta< \eta_0$  and take $\delta:= \frac{\eta}{\eta'}-1.  $
Fix $N$ with  $$\int e^{\eta d(x, x_0)} \ d \nu_j(x) <N$$
for all $j$.
Using Markov's inequality,  for all  $M>0$ and every $j$ we have
$$\nu_j\{ x : e^{\eta d(x,x_0)} > M\} \le  N/M $$
so $$\nu\{ x : e^{\eta d(x_0, x)} > M\} \le  N/M.$$


Then, for the limit measure $\nu$, we have
\begin{align*}
\int_{G/\Gamma} e^{\eta' d(x_0, x)} \ d \nu(x)
&=
\int_0^\infty
\nu\{ x : e^{\eta'  d(x_0, x)} \ge M\} \ d M\\
&=
\int_0^\infty
\nu\{ x : \left(e^{\eta  d(x_0, x)}\right)^{1/(1+\delta)} \ge {M}   \} \ d M\\
&= \int_0^\infty \nu\{ x : e^{\eta  d(x_0, x)} \ge {M}^{1+\delta}   \} \ d M\\
&\le 1 + \int_1^\infty  \frac{N}{{M}^{1+\delta} }  \ d M.\qedhere
\end{align*}
\end{proof}

\subsection{Averaging certain measures on $\Sl(m,\R)/\Sl(m,\Z)$}
Take  $\{\alpha_1, \dots, \alpha_m\}$ to be the standard set of simple positive roots of $\Sl(m,\R)$:
	$$\alpha_j (\diag(e^{t_1}, \dots, e^{t_m})) = t_j-t_{j+1}.$$

Let $H_1 $  be the analytic subgroup of $\Sl(m,\R)$ whose Lie algebra is generated by roots spaces associated to  $\{\pm \alpha_1\}$ and let  $H_2$ be the analytic subgroup of $\Sl(m,\R)$ whose Lie algebra is  generated by roots spaces associated to $\{\pm \alpha_3, \dots, \pm \alpha_n\}$.  We have  $H_1 \cong \Sl(2,\R)$ and $H_2 \cong \Sl(m-2, \R)$.  Then $H = H_1\times H_2\subset \Sl(m,\R)$ is the subgroup of all matrices of the form
$$\left(\begin{array}{cc}B& 0 \\0 & C\end{array}\right)$$
where $\det(B) = \det(C) =1$.

 We let $A'$ be the  the co-rank-1 subgroup $A'\subset A$  of the Cartan subgroup $A$  given by $A'= A\cap H$.
Let $\delta= \alpha_1 + \dots + \alpha_n$ be the highest positive  root.
\begin{proposition}
\label{proposition:averaging}
Let $\mu$ be any $H$-invariant probability on $\Sl(m,\R)/\Sl(m,\Z)$.  Let $\beta'= \alpha_2$ or $\beta' = \delta$ and let $\hat \beta = - \alpha_2$ or $\hat \beta =- \delta$.

Then $ U^{\beta'}\ast \mu$ is $H$-invariant and  $$U^{\hat \beta}\ast  U^{\beta'}\ast \mu$$ is the Haar measure on $G/\Gamma$.
\end{proposition}
\begin{proof}
We have that $\mu$ is $A'$-invariant.  Let $\mu'= U^{\beta'}\ast \mu$ and note that $\mu'$ remains $H$- and $A'$-invariant.

 {\bf Case 1(a) : $\beta' = \alpha_2$.} Consider first the case that $\beta' = \alpha_2$.  Then $\mu'$ remains invariant under $U^{-\alpha_1}$ and $U^{-\alpha_j}$ for all $3\le j\le n$ since these roots commute with $\beta'$.
By Theorem \ref{thm:ratner}\ref{ratner3}   we have that $\mu'$ is also invariant under
$U^{\alpha_1}$ and $U^{\alpha_j}$ for all $3\le j\le n$.  Taking brackets, $\mu'$ is invariant under $U^\beta$ for every positive root $\beta\in \Sigma_+$.

 {\bf Case 1(b) : $ \beta' = \delta$.}
Consider now the case that $\beta' = \delta$.  Then $\mu'$ remains invariant under $U^{\alpha_1}$ and $U^{\alpha_j}$ for all $3\le j\le n$ since these roots commute with $\delta$.
By   Theorem \ref{thm:ratner}\ref{ratner3}  we have that $\mu'$ is also invariant under
$U^{-\alpha_1}$ and $U^{-\alpha_j}$ for all $3\le j\le n$.  Taking brackets, $\mu'$ is invariant under $U^\beta$ for every positive root $\beta$ of the form $\delta - \alpha_n - \alpha_{n-1} - \dots - \alpha_j = \alpha_1 +\dots + \alpha_{j-1}$ for each $j\ge 3$.  In particular, $\mu'$ is invariant under $U^{ \alpha_1 +\alpha_2}$ and hence also  invariant under $U^{ \alpha_2}$.  In particular $\mu'$ is invariant under $U^\beta$ for every positive root $\beta\in \Sigma_+$.

Note that in either case, we have that $\mu'$ is invariant under $U^\beta$ for every positive root $\beta\in \Sigma_+$.

Let $\hat \mu = U^{\hat \beta}\ast \mu'$.

 {\bf Case 2(a) : $\hat \beta = -\alpha_2$.}
If $\hat \beta = -\alpha_2$, then $\hat \mu$ remains invariant under $U^{\alpha_1}$ and $U^{\alpha_j}$ for all $3\le j\le n$.  Note additionally $\hat \mu$ remains invariant under the highest-root group $U^{\delta}$.  Again, by  Theorem \ref{thm:ratner}\ref{ratner3}  we have that $\hat \mu$ is also invariant under
$U^{-\alpha_1}$ and $U^{-\alpha_j}$ for all $3\le j\le n$.  In particular $\hat \mu$ is also invariant under  $U^\beta$ for every negative  root $\beta\in \Sigma_-$.   It follows as in Case 1(b) that $\hat \mu$ is invariant under $U^{\alpha_2}$ and hence invariant under $U^\beta$ for every positive  root $\beta\in \Sigma_+$.  Thus $\mu$ is $G$-invariant.

 {\bf Case 2(b) : $\hat \beta = -\delta$.}
If $\hat \beta = -\delta$, then $\hat \mu$ remains invariant under $U^{-\alpha_1}$ and $U^{-\alpha_j}$ for all $3\le j\le n$.  Note additionally $\hat \mu$ remains invariant under   $U^{\alpha_2}$.  Again,  we have that $\hat \mu$ is also invariant under
$U^{\alpha_1}$ and $U^{\alpha_j}$ for all $3\le j\le n$.  In particular $\hat \mu$ is also invariant under  $U^\beta$ for every positive  root $\beta\in \Sigma_+$.   As in Case 1(b) that $\hat \mu$ is invariant under $U^{-\alpha_2}$ and hence invariant under $U^\beta$ for every negative  root $\beta\in \Sigma_-$.  Thus $\mu$ is $G$-invariant.
\end{proof}

\def \top{\mathrm{top}}


\subsection{Lyapunov exponents for unbounded cocycles}
\label{subsection:Lyap}
\def\calE{\mathcal E}
Let $(X,d)$ be a second countable, complete metric space.  We moreover   assume the metric $d$ is proper.  Let $G$ act continuously  on $X$.

Let $\calE\to X$ be a continuous,  finite-dimensional  vector bundle  equipped with a norm  $\|\cdot\|$.  A linear cocycle over the $G$-action on $X$ is an action $\calA\colon G\times \calE\to \calE$ by vector-bundle automorphisms that projects to the $G$-action on $X$.
We write $\calA(g,x)$ for the linear map between Banach spaces $\calE_x$ and $\calE_{g\cdot x}$.  By the norm of $\calA(g,x)$ we mean the operator norm and the conorm is $m(\calA(g,x))= \|\calA(g,x)^{-1}\|^{-1}$.
We say that $\calA$ is \emph{tempered} with respect to the metric $d$ if there is a $k\ge0$ such that  for any compact set $K\subset G$ and base point $x_0\in X$ there is $C>1$ so that $$\sup_{g\in K} \|\calA(g,x)\| \le Ce^{k d(x,x_0)}$$ and
 $$\inf_{g\in K} m(\calA(g,x)) \ge \frac{1}{C}e^{-k d(x,x_0)}$$
 where $\|\cdot\|$ denotes the operator norm and $m(\cdot)$ denotes the operator conorm applied to  linear maps between Banach spaces $\calE_x$ and $\calE_{g\cdot x}$.  

%
%
If $\mu$ is a probability measure on $(X,d)$ with exponentially small mass in the cusps, it follows that the function  $x\mapsto d(x,x_0)$ 
is $L^1(\mu)$ whence  we immediately  obtain the following.  
\begin{claim}\label{lem:foobar}
Let $\mu$ a probability  measure on $X$ with exponentially small mass in the cusps.   Suppose that $\calA$ is tempered.  Then for any compact $K\subset G$, the functions   $$x\mapsto  \sup_{s\in K} \log \left \| \calA(s, x)\right \| ,\quad \quad x\mapsto    \inf_{s\in K}\log m\left( \calA(s, x)\right)  $$ are $L^1(\mu)$.
\end{claim}
%
%
%
%

Given $s\in G$ and an $s$-invariant Borel probability measure $\mu$ on $X$ we define the \emph{average leading} (or \emph{top})  \emph{Lyapunov exponent of $\calA$} to be
\begin{equation} \label{eq:toplyap} \lambda_{\top,s,\mu, \calA} := \inf _{n\to \infty} \frac 1 n  \int \log \|\calA(s^n, x)\| \ d \mu (x).\end{equation}
From the   integrability of the function $x\mapsto  \log \|\calA(s, x)\|  $   we obtain   the finiteness of Lyapunov exponents.
\begin{corollary}
\label{corollary:finite}
For $s\in G$ and   $\mu$   an $s$-invariant probability measure on $X$ with exponentially small mass in the cusps, if  $\calA$ is tempered
then the average leading   Lyapunov exponent $\lambda_{\top,s,\mu, \calA}$ of $\calA$ is finite.
\end{corollary}
Note that for an $s$-invariant measure $\mu$,  the sequence  $  \int \log \|\calA(s^n, x)\| \ d \mu (x)$ is subadditive whence the infimum in \eqref{eq:toplyap} maybe replaced by a limit.

As in the case of bounded continuous linear cocycles, we  obtain  upper-semicontinuity of leading Lyapunov exponents   for continuous tempered cocycles when restricted to families of  measures with uniformly exponentially  small measure in the cusp.

\begin{lemma}\label{claim:jjjiiilllkkk}
Let $\calA$ be a tempered cocycle.  Given $s\in G$ suppose  the restriction of the cocycle  $\calA\colon G\times\calE\to \calE$ to the action of $s$ is continuous.

Then---when restricted to a set of $s$-invariant Borel probability measures with  uniformly exponentially small mass in the cusps---the function  $$\mu \mapsto \lambda_{\top,s,\mu, \calA} $$ is upper-semicontinuous with respect to the weak-$*$ topology.
\end{lemma}

\begin{proof}

Let $\calM = \{\mu_\zeta\}_{\zeta\in \calI}$ be a family of $s$-invariant Borel probability measures with  uniformly exponentially small mass in the cusps.  As the pointwise infimum of continuous functions is upper-semicontinuous, is enough to show that the function $$\calM\to \R,\quad \quad \mu\mapsto \int \log \|\calA(s^n, x)\| \ d \mu (x)$$
is  continuous with respect to the weak-$*$ topology for each $n$.  As the weak-$*$ topology is first countable, it is enough to show $\mu\mapsto \int \log \|\calA(s^n, x)\| \ d \mu (x)$ is sequentially continuous.

Let $\mu_j \to \mu_\infty$ in $\calM$.  Given $M>0$, fix a continuous
$\psi_M\colon X\to [0,1]$ with $$ \text{$ \psi_M(x) = 1$ if $d(x,x_0) \le M$ and $\psi_M(x) = 0$ if $d(x,x_0)\ge M+ 1$}.$$

As we assume our metric   is proper,  $x\mapsto \psi_M (x)\log \|\calA(s^n, x)\|$ is a bounded continuous function whence  $$\int \log\psi_M(x)  \log \|\calA(s^n, x)\|\  d \mu_j (x) \to \int\psi_M(x) \log \|\calA(s^n, x)\|  \ d \mu_\infty  (x).$$
Moreover, there are $C>1, k\ge 1,$ and $\eta>0$ such that for all $x\in X$ and $\mu_\zeta\in \calM$
$$- \log C- {k d(x,x_0)} \le \log \|\calA(s^n, x)\| \le \log C+ {k d(x,x_0)}, $$ and $$ \int e^{\eta d(x, x_0)} \ d \mu_\zeta(x)\le C.$$
In particular, $$\mu_\zeta(\{x: d(x,x_0) \ge   M\})\le C e^{-\eta   M}.$$

Thus for any $\mu_\zeta\in \calM$,  we have
\begin{align*}
\int  \big | \log \|\calA(s^n, x)\| &- \psi_M(x)\log \|\calA(s^n, x)\|   \big|  \ d \mu_\zeta(x) \\
& \le  \int_{\{x: d(x,x_0) \ge   M \}}\big | \log \|\calA(s^n, x)\|  - \psi_M(x)\log \|\calA(s^n, x)\|   \big|   \ d \mu_\zeta(x) \\
& \le  \int_{\{x: d(x,x_0) \ge   M \}}\big | \log \|\calA(s^n, x)\|   \big|   \ d \mu_\zeta(x) \\
    &\le  \int_{\{x: d(x,x_0) \ge   M\}}  \log C+ {k d(x,x_0)}    \ d \mu_\zeta(x)  \\
    &\le (\log C  ) C e^{-\eta  M}   + k \int_{\{x: d(x,x_0) \ge   M\}}   { d(x,x_0)}   \ d \mu_\zeta(x)  \\
        &\le (\log C   + k M) C e^{-\eta M}  + k\int_{\ell =  M} ^\infty \mu_\zeta\{ x:    { d(x,x_0)} \ge \ell\} \ d \ell\\
                &\le (\log C +k  M) C e^{-\eta M}  + k\int_{\ell =  M} ^\infty  Ce^{-\eta \ell}  \ d \ell\\
                                &\le (\log C + k M) C e^{-\eta M}  + k\frac{C e^{\eta  (- M)}}{\eta }.
\end{align*}
It follows that given $\epsilon>0$ there is $M$ so that
$$\int  \big | \log \|\calA(s^n, x)\| - \psi_M(x)\log \|\calA(s^n, x)\|   \big|  \ d \mu_\zeta(x) \\  \le \epsilon$$
for all $\mu_\zeta\in \calM$.

In particular, taking $M$ and $j$ sufficiently large we have
\begin{align*}
\Big|\int& \log \|\calA(s^n, \cdot )\|   \ d \mu_\infty  - \int \log \|\calA(s^n, \cdot )\| \ d \mu_j  \Big|
\\
&\le
\int \big|\log \|\calA(s^n, \cdot )\|  - \psi_M \log \|\calA(s^n, \cdot )\|\big|    \ d \mu_\infty  \\ &\quad\quad +
\Big|\int \psi_M  \log \|\calA(s^n, \cdot )\| \ d \mu_j  - \int \psi_M  \log \|\calA(s^n, \cdot )\| \ d \mu_\infty  \Big| \\ & \quad \quad +
\int \big|\log \|\calA(s^n, \cdot )\|  - \psi_M \log \|\calA(s^n, \cdot )\| \big|   \ d \mu_j    \\
&\le  3 \epsilon. 
\end{align*}
Sequential continuity then follows.
\end{proof}

\subsection{Lyapunov exponents under averaging and limits}
We now consider the behavior of the top Lyapunov exponent $\lambda_{\top, s, \mu, \calA}$ as we average an $s$-invariant probability measure $\mu$ over an amenable subgroup of $G$ contained in the centralizer of $s$.
\begin{lemma}
\label{lemma:averaging}
Let $s\in G$ and let $\mu$ be an $s$-invariant probability measure on $X$ with exponentially small mass in the cusps.  Let $\calA\colon G\times \calE\to \calE$ be a tempered continuous cocycle.

For any amenable subgroup $H\subset C_G(s)$ and any   \Folner sequence of precompact sets $F_n$ in $H$,  if the family
$\{F_n \ast \mu\}$ has uniformly exponentially small mass in the cusps  then for any subsequential limit
$\mu'$ of $\{F_n \ast \mu\}$ we have
  $$\lambda_{\top, s, \mu, \calA} \le \lambda_{\top, s, \mu', \calA}.$$
\end{lemma}

\begin{proof}
{\blue First note that  Lemma \ref{lem:pastrydough} implies the family $\{F_n \ast \mu\}\cup \{\mu'\}$  has uniformly exponentially small mass in the cusps.}    Note also that for every $m$, the measure $F_m \ast \mu$ is $s$-invariant.

We first claim that  $\lambda_{\top, s, F_m \ast \mu, \calA} =  \lambda_{\top, s, \mu, \calA}$ for every $m$.
For   $t \in H$ define $c_t (x)=   \sup\{\|\calA(t , x)\| , m(\calA(t, x))\inv \} $ and let $c_m(x) = \sup _{t\in F_m } c_t(x)$.
As $F_m$ is precompact, from  Claim \ref{lem:foobar} we have  that $\log c_m \in L^1(\mu)$.  

 For  $x \in M$ and   $t \in F_m$, the cocycle property and subadditivity of norms yields
\begin{align*}
\log \|\calA(s^n, tx)\| &\leq  \log \|\calA(t\inv, tx)\| + \log \|\calA(s^n, x)\| + \log \|\calA(t , s^nx)\|  \\
&=  \log \|\calA(t, x)\inv\| + \log \|\calA(s^n, x)\| + \log \|\calA(t , s^nx)\|  \\  
&\leq \log c_m(x) + \log c_m (s^n(x)) + \log \|\calA(s^n, x)\|. 
\end{align*}
Using that $\mu$ is $s$-invariant, we have for every $n$ that
\begin{align*}
\int \log &\|\calA(s^n, x) \| \ d (F_m\ast \mu )(x) = \dfrac 1 {|F_m|}\int_{F_m} \int \log \|\calA(s^n, x)\| \ dt*\mu (x) \ d t \\
&= \dfrac 1 {|F_{m}|}\int_{F_{m}} \int \log \|\calA(s^n, tx)\| \ d\mu (x) \ d t \\
&\leq \dfrac 1 {|F_{m}|}\int_{F_{m}}  \bigg(\int  \log c_m(x) + \log c_m (s^n(x)) +  \log \|\calA(s^n, x)\|  \ d \mu (x) \bigg) \ d t \\
&\leq 2\int  \log  c_m(x)   \ d \mu(x) + \int \log \|\calA(s^n, x)\| \ d \mu{(x)}
\end{align*}
Dividing by $n$ yields $\lambda_{\top, s, F_m \ast \mu,\calA}\le \lambda_{\top, s, \mu, \calA}$.   The reverse  inequality is similar.

The inequality then follows from the  upper-semicontinuity  in Lemma \ref{claim:jjjiiilllkkk}.
\end{proof}

Consider now any $Y\in \lieg$ with $\|Y\|=1$, a point  $x\in X$,  and $t>0$.
The     empirical measure $\eta(Y,t,x)$ along the orbit $\exp (sY) x$ until time $t$ is the measure defined as follows: given a bounded continuous $\phi\colon X\to \R$, the integral of $\phi$ with respect to the empirical measure  $\eta(Y,t,x)$ is
$$
\int \phi \ d \eta (Y, t, x) := \frac{1}{t} \int_0^{t} \phi \big(\exp (sY) \cdot x\big) \ d s.$$
Similarly, given a probability measure $\mu$ on $X$,   the empirical distribution $\eta(Y,t,\mu)$ of $\mu$ along the orbit of $\exp (sY) $ until time $t$ is defined as 
$$
\int \phi \ d \eta(Y,t, \mu):= \frac{1}{t} \int_X \int_0^{t} \phi \big(\exp (sY) \cdot x\big) \ d s \ d \mu(x).$$

Consider now sequences  $Y_n\in \lieg$ with $\|Y_n\|=1$ and  $t_n>0$.   For part \ref{lazylemmac} of the following lemma, we add an additional  assumption that the action of $G$ on $(X,d)$ has \emph{uniform displacement:} for any compact $K\subset G$ there is $C'$ such that for all $x\in X$ and $g\in K$, $$d(g\cdot x, x )\le C'.$$

\begin{lemma} 
\label{lemma:firstexponents}
Suppose the action of $G$ on $(X,d)$ has uniform displacement and let $\calA\colon G\times \calE\to \calE$ be a tempered continuous cocycle.

Let $Y_n\in \lieg$ and  $t_n\ge 0$ be  sequences with  $\|Y_n\|=1$ for all $n$ and $t_n\to \infty$. Let $\mu_n$ be a sequence of Borel probability measures on $X$ and define $\eta_n :=  \eta(Y_n,t_n,\mu_n)$ to be the empirical distribution of $\mu_n$ along the orbit of $\exp (sY_n)$ for $0\le s\le t_n$.
 Assume that
\begin{enumerate}
\item the family of empirical distributions  $\{\eta_n\} $ defined above has uniformly exponentially small mass in the cusps; and
\item $\int \log  \|\calA ( \exp (t_nY_n),x) \|  \ d \mu_n(x) \geq \epsilon t_n$.
\end{enumerate}
Then
\begin{enumlemma}
\item \label{lazylemmaa} the family $\{\eta_n\} $ is pre-compact;
\item \label{lazylemmab}for any subsequential limit $Y_\infty = \lim_{j\to \infty}  Y_{n_j},$  any subsequential limit $\eta_\infty $ of  $\{ \eta_{n_j}\}$ is invariant under the 1-parameter subgroup $\{ \exp (tY_\infty): t\in \R\}$;
\item \label{lazylemmac}$\lambda_{\top,  \exp (Y_\infty), \eta_\infty,\calA}\ge\epsilon>0$.
\end{enumlemma}
\end{lemma}

\begin{proof}[Proof of Lemma \ref{lemma:firstexponents} \ref{lazylemmaa} and  \ref{lazylemmab}]
As in the proof of Lemma \ref{claim:jjjiiilllkkk}, from the assumption that $\{\eta_n\}$ has uniformly exponentially small mass in the cusps we obtain uniform bounds $$\eta_n (\{x: d(x,x_0) \ge \ell\})\le C e^{-\eta \ell}$$ for all $n$.  Combined with the properness of $d$, this establishes uniform tightness of the family of  measures $\{\eta_n\}$ and  \ref{lazylemmaa} follows.

For  \ref{lazylemmab},    let  $\phi \colon X \to \R$ be a compactly supported  continuous function.  Then for any $s>0$
\begin{align*}
\int_{X} \phi \circ \exp (sY_\infty) -  \phi \ d \eta_n
&= \int_{X} \phi \circ \exp (sY_\infty)  -  \phi \circ \exp(s Y_n) \ d\eta_n\\
&+ \int_{X}  \phi \circ \exp(sY_n) - \phi \ d\eta_n
\end{align*}

\noindent The first integral converges to zero as the functions $\phi \circ \exp (wY_\infty)  -  \phi \circ \exp(w Y_n)$ converges uniformly to zero in $n$ for fixed $w$.  The second integral clearly converges to zero  since for $t_n\ge s$ we have
\begin{align*}
\int_{X}  \phi &\circ \exp(s Y_n) - \phi \ d\eta_n=
\frac{1}{t_n} \int_0^{t_n} \int_X \phi \left( \exp\left((s+t) Y_n\right)  x \right) - \phi \left(\exp (tY_n) x \right) \ d \mu_n(x) \ d t\\
&=
\frac{1}{t_n}\left[ - \int_0^{s} \int _X\phi \left( \exp\left(t Y_n\right)  x \right) \mu_n(x) \ d t + \int_{t_n}^{t_n+s}  \int_X \phi \left( \exp\left(t Y_n\right)  x \right) \ d \mu_n(x) \ d t\right]
\end{align*}
which converges to 0 as $t_n\to \infty$ as $\phi$ is bounded.
\end{proof}


The proof of Lemma \ref{lemma:firstexponents}\ref{lazylemmac} is quite involved. It is the analogue in the non-compact
setting of \cite[Lemma 3.6]{BFH};  we recommend the reader read the proof of of \cite[Lemma 3.6]{BFH} first.  Two technical
complications arise in the proof of Lemma \ref{lemma:firstexponents}\ref{lazylemmac}.
 First, we must control for ``escape of Lyapunov exponent'' as our cocycle is unbounded.    Second, in \cite{BFH} it was sufficient to  consider the average  of Dirac masses $\delta_{x_n}$ along a single orbit $\exp(sY_n)x_n$; here we average measures $\mu_n$  along an  orbit of $\exp (sY_n).$

To prove Lemma \ref{lemma:firstexponents}\ref{lazylemmac} we first introduce a number of {\blue standard} auxiliary objects.
 Let $\P \calE\to X$ denote the projectivization of the  tangent bundle    $\calE\to X$.
 We represent a point in $\P\calE$ as $(x,[v])$ where $[v]$ is an equivalence class of non-zero vectors in the fiber $\calE(x)$.
 For each $n$, let $\sigma_n\colon X\to \calE\sm\{0\}$ be a nowhere vanishing Borel section such that
$$\|\calA (\exp (t_nY_n),x ) (\sigma_n(x))\| \|(\sigma_n(x))\|\inv = \|\calA (\exp (t_nY_n),x)\|$$
for every $x\in X$.
The $G$-action on $\calE$ by vector-bundle automorphisms induces  a natural $G$-action on $\P \calE$ which restricts to projective transformations between each fiber and its image.    For each $n$, let $\td \eta_n$ be the probability measure  on $\P\calE$ given as follows: given a bounded continuous $\phi\colon \P \calE\to \R$ define
$$\int _{\P \calE} \phi  \ d \td \eta_n  :=
\frac{1}{t_n} \int_0^{t_n} \int _X \phi \big(\exp (tY_n) \cdot (x,  [\sigma_n(x)] )\big )\ d \mu_n(x) \ d t.$$
We have that $\td \eta_n$ projects to $ \eta_n$ under the natural projection $\P\calE\to X$; moreover, if $\eta_{j_k}$ is a subsequence converging to $\eta_\infty,$ then   any weak-$*$ subsequential limit $\td \eta_\infty$ of $\{ \td \eta_{n_{j_k}}\}$   projects to $ \eta_\infty$.

Define  $\Phi\colon \lieg \times \P\calE\to \R$ by  $$\Phi\big(Y,(x, [v])\big) := \log \left(\left\| \calA\big(\exp
(Y), x\big)v\right\|\|v\|\inv\right).$$
Note for each fixed $Y\in \lieg$ that $\Phi$ satisfies a cocycle property:
\begin{equation}\label{eq:cocycle}\Phi\big((s+t)Y,(x, [v])\big)= \Phi\big(tY,(x, [v])\big) +  \Phi\big(sY,\exp(tY)\cdot(x, [v])\big) \end{equation}

By hypothesis, there are $C>1$, $ k\ge 1$, and $\eta>0$ such that $$\int e^{\eta d(x,x_0)} \ d \eta_n \le C$$ for all $n$ and
$$ \frac{1}{C}  e^{-k d(x,x_0)} \le \left \|\calA(\exp (Y),x)v\right\| \|v\|\inv  \le  C  e^{k d(x,x_0)}$$ for all $(x,[v])\in \P\calE$ and $Y\in \lieg$ with $\|Y\|\le 1.$


For each $n$, let $$M_n(x) = \sup _{0\le t\le t_n}\left\{d\left( \big(\exp (tY_n) x\big),x_0\right)\right\}.$$ 
As we assume the  $G$-action on $(X,d)$ has uniform displacement, take $$C_1 = \sup_{\|Y\|\le 1, x\in X} \{d(\exp (Y)\cdot x, x)\}.$$
{\blue
We have  $$ \frac{1}{t_n} \int_0^{t_n} \int_X e^{\eta d\left( (\exp (tY_n) x ,x_0\right)} \ d \mu_n(x) \ d t= \int e^{\eta d(x,x_0)} \ d \eta_n \le C.$$
If $t_n\ge 1$ then for every $x$ there is an interval $I_x\subset [0,t_n]$ of length $1$ on which $$d\left( (\exp (tY_n) x ,x_0\right)\ge (M_n (x)-  C_1)$$ for all $t\in I_x.$
It follows that
$$\int_X e^{\eta (M_n (x)-  C_1)}   \ d \mu_n(x) \le    \int_X \int_{I_x} e^{\eta d\left( (\exp (tY_n) x ,x_0\right)} \ d t  \ d \mu_n(x) \le  Ct_n.  $$ 
}

By Jensen's inequality we have
$$\int_X \eta ( M_n(x) - C_1) \ d \mu_n(x) \le \log \int_X e^{\eta (M_n (x)-  C_1)}  d \mu_n (x)$$
whence
 $$\int M_n (x) \ d \mu_n(x) \le \eta\inv( \log  C  + \log t_n )  + C_1 =: \eta\inv \log t_n +  C_2.$$

Since $\|Y_n \| = 1,$ 
we have \begin{equation}\label{eq:autocracy} \begin{aligned}
 \sup _{0\le t\le t_n, 0\le s\le 1}  \int_X &\left| \Phi(sY_n, \exp(t Y_n)\cdot (x, [\sigma_n(x)]))\right |  \ d \mu_n(x)\\
&\le \int_X \sup _{0\le t\le t_n, 0\le s\le 1} \left  | \Phi(sY_n, \exp(t Y_n)\cdot (x, [\sigma_n(x)]))\right |  \ d \mu_n(x)\\
&\le \int |\log C| + k M_n (x) \ d \mu_n(x)\\
&\le |\log C |+ k ( \eta\inv \log t_n +  C_2) \\
&=:   k  \eta\inv \log t_n +  C_3.
\end{aligned}\end{equation}
In particular, we have
 \begin{align*}\frac{1}{t_n} & \int_X \log \|\calA(\exp(t_nY_n), x) \| \ d \mu_n(x) \\
 & = \frac{1}{t_n}\int_X \Phi(t_nY_n, (x,[\sigma_n(x)])) \ d \mu_n(x) \\
 & = \frac{1}{t_n}\int_X\Phi(\lfloor t_n\rfloor Y_n, (x,[\sigma_n(x)]))\ d \mu_n(x)  \\ &\quad  +
 \frac{1}{t_n}\int_X \Phi((t-\lfloor t_n\rfloor) Y_n, \exp (\lfloor t_n\rfloor Y_n)\cdot (x,[\sigma_n(x)]))\ d \mu_n(x) .
 \end{align*}
 Since
 $$\left|\frac{1}{t_n} \int_X \Phi((t-\lfloor t_n\rfloor) Y_n, \exp (\lfloor t_n\rfloor Y_n)\cdot (x,[\sigma_n(x)]))\ d \mu_n(x) \right|\le \frac{1}{t_n}( k  \eta\inv \log t_n +  C_3)$$ goes to 0 as $t_n\to \infty$
 it follows that
 \begin{equation}
 \begin{aligned}\label{eq:thiseq}\liminf_{n\to \infty} \int_X  \frac{1}{t_n}&\Phi(\lfloor t_n\rfloor Y_n, (x,[\sigma_n(x)]))  \ d \mu_n(x)  \\&= \liminf_{n\to \infty} \frac{1}{t_n} \int_X  \log \|\calA(\exp(t_nY_n), x) \| \ d \mu_n(x)\\& \ge \epsilon >0.
 \end{aligned}\end{equation}

With the above objects and estimates we complete the proof of  Lemma \ref{lemma:firstexponents}.
\begin{proof}[Proof of Lemma \ref{lemma:firstexponents} \ref{lazylemmac}]
Consider first the expression $\int \Phi (Y_n, \cdot )\  d \td \eta_n.$  We have
\begin{align*}
\int \Phi  & (Y_n, \cdot ) \ d \td \eta_n
\\& =  \frac{1}{t_n}\int_0^{t_n} \int_X \Phi  \big(Y_n,\exp (tY_n) \cdot (x,  [\sigma_n(x)] )\big ) \ d\mu_n(x) \ d t
\\&= \frac{1}{t_n}  \int_0^{\lfloor t_n\rfloor}\int_X  \Phi    \big(Y_n,\exp (tY_n) \cdot (x,  [\sigma_n(x)] )\big )  \ d\mu_n(x) \ d t \\
&\quad
+ \frac{1}{t_n} \int_ {\lfloor t_n\rfloor} ^{t_n}\int_X \Phi    \big(Y_n,\exp (tY_n) \cdot (x,  [\sigma_n(x)] )\big ) \ d\mu_n(x)   \ d t
\end{align*}
Note that the contribution of the second integral is bounded by
$$\left| \frac{1}{t_n} \int_ {\lfloor t_n\rfloor} ^{t_n}\int_X \Phi    \big(Y_n,\exp (tY_n) \cdot (x,  [\sigma_n(x)] )\big ) \ d\mu_n(x)   \ d t \right|
\le \frac{1}{t_n}(k\eta\inv \log t_n +  C_3)$$
which goes to zero as $t_n\to \infty$.

 Repeatedly applying the cocycle property \eqref{eq:cocycle} of $\Phi(Y_n, \cdot)$ we have  for $t_n\ge 1$ that
\begin{align*}
  \frac{1}{t_n}  \int_X  \int_0^{\lfloor t_n\rfloor} \Phi  & \big(Y_n,\exp (tY_n) \cdot (x,  [\sigma_n(x)] )\big ) \ d t \ d \mu_n(x)
  \\&= \frac{1}{t_n} \int_X  \int_0^{1}   \Phi    \big({\lfloor t_n\rfloor} Y_n, \exp (tY_n) \cdot (x,  [\sigma_n(x)] )\big ) \ d t\ d \mu_n(x)
  \\&= \frac{1}{t_n}  \int_X \int_0^{1}  \Big( \Phi    \big({\lfloor t_n\rfloor} Y_n,  (x,  [\sigma_n(x)] )\big )
  -  \Phi    \big( t Y_n,  (x,  [\sigma_n(x)] )\big )
  \\ & \quad\quad\quad +     \Phi    \big( t Y_n, \exp (\lfloor t_n\rfloor Y_n)\cdot  (x,  [\sigma_n(x)] )\big )
  \Big)
   \ d t\ d \mu_n(x)
   \\& =  \frac{1}{t_n}   \int_X \Phi    \big({\lfloor t_n\rfloor} Y_n,  (x,  [\sigma_n(x)] )  \ d \mu_n(x) +
 \frac{1}{t_n} \int_X \int_0^{1}  \Big(    -  \Phi    \big( t Y_n,  (x,  [\sigma_n(x)] )\big )   \\ & \quad\quad\quad +      \Phi    \big( t Y_n, \exp (\lfloor t_n\rfloor Y_n)\cdot  (x,  [\sigma_n(x)] )\big )
  \Big) \ dt\ d \mu_n(x)
%
%
\end{align*}
From \eqref{eq:autocracy}, the contribution of the second and third integrals is bounded by
\begin{align*}
 &\left|\frac{1}{t_n}\int_X  \int_0^{1}  \Big(    -  \Phi    \big( t Y_n,  (x,  [\sigma_n(x)] )\big )
      +     \Phi    \big( t Y_n, \exp (\lfloor t_n\rfloor Y_n)\cdot  (x,  [\sigma_n(x)] )\big )
  \Big)  \ d t  \ d \mu_n(x)
 \right|
 \\&\quad \quad \quad \quad \le \frac{1}{t_n}  \int_0^{1}  2 ( k  \eta\inv \log t_n +  C_3)\ dt
 \\&\quad \quad \quad \quad = \frac{1}{t_n}    2 ( k  \eta\inv \log t_n +  C_3)
\end{align*}
which tend to zero as $t_n\to \infty$.  We then conclude from \eqref{eq:thiseq} that
\begin{equation}
\label{eq:keyless}
\liminf_{n\to \infty} \int \Phi   (Y_n, \cdot ) \ d \td \eta_n = \liminf_{n\to \infty}    \frac{1}{t_n} \int_X    \Phi    \big({\lfloor t_n\rfloor} Y_n,  (x,  [\sigma_n(x)] ) \ d \mu_n(x)
\ge \epsilon >0.
\end{equation}

To complete the proof of \ref{lazylemmac}, for  $M>0$     take $\psi_M\colon X\to [0,1]$ continuous with $$ \text{$ \psi_M(x) = 1$ if $d(x,x_0) \le M$ and $\psi_M(x) = 0$ if $d(x,x_0)\ge M+ 1$}.$$
Let $\Psi_M\colon \P\calE \to [0,1]$ be $$\Psi_M(x,[v]) = \psi_M(x).$$
and define $\Phi_M \colon \lieg\times \P\calE\to \R$ to be
$$\Phi_M\big(Y,(x, [v])\big) :=\Psi_M(x, [v]) \Phi\big(Y,(x, [v])\big).$$

As the family $$  \mathcal N= \{\eta_n\} \cup\{\eta_\infty\}$$ has uniformly exponentially small mass in the cusps we have
$$\int e^{\eta d(x,x_0)} d \hat \eta <C$$
 and hence $\hat \eta\{x: d(x,x_0)\ge \ell \} \le Ce^{-\eta\ell}$
for all $\hat \eta \in   \mathcal N$.
It follows for all $\td  \eta\in  \{\td \eta_n\} \cup\{\td \eta_\infty\}$  that---letting $\hat  \eta\in \mathcal N$ denote the image of $ \td  \eta$ in $X$---we have for any $Y\in \lieg$  with $\|Y\|\le 1$ that
\begin{align*}
	\int_{ \P\calE} &|\Phi(Y, \cdot ) - \Phi_M(Y, \cdot )| \ d  \td  \eta
		\\&= \int_{\{(x,[v]\in \P\calE: d(x,x_0)\ge  M\}} |\Phi (Y, \cdot ) - \Phi_M (Y, \cdot ) | \ d  {\td \eta}
				\\&\le \int_{\{(x,[v]\in \P\calE: d(x,x_0)\ge  M\}}  |\Phi (Y, \cdot ) | \ d  {\td \eta}
		\\&\le \int_{\{x\in X: d(x,x_0)\ge  M\}}  \log(C) + {k d(x,x_0)} \ d \hat  \eta
	\\& \le {\blue  (\log C + k M)Ce^{-\eta M}   + k \int_{ M}^\infty  \hat \eta \{ x:  { d(x,x_0)}\ge \ell\} \  d \ell}
	  \\ &\le (\log C + k M) C e^{-\eta M}  + k\frac{C e^{\eta  (- M)}}{\eta }.
		\end{align*}
%
In particular, given any $\delta>0$, by taking $M>0$ sufficiently large we may ensure that $$\int_{ \P\calE} |\Phi(Y, \cdot ) - \Phi_M(Y, \cdot )| \ d  \td  \eta\le \delta$$ for any   $$\td  \eta\in  \{\td \eta_n\} \cup\{\td \eta_\infty\}.$$

Since   the restriction of $\Phi_M$ to $\{ Y\in \lieg : \|Y\|\le 1\} \times \P\calE$ is compactly supported, it is uniformly continuous whence
$$\int  \Phi_M(Y_n, \cdot ) \ d \td \eta_n - \Phi_M(Y_\infty, \cdot )\ d \td \eta_\infty \to 0$$
as $n\to \infty.$
In particular given $\delta>0$ we may take $M$ and $n$ sufficiently large so that
\begin{align*}
\Big|\int_{ \P\calE}  &\Phi(Y_n, \cdot ) \ d \td \eta_n  - \int _{\P \calE} \Phi(Y_\infty, \cdot )  \ d \td  \eta_\infty\Big|
\\&\le  \int_{ \P\calE}  \left|  \Phi(Y_n, \cdot ) - \Phi_M(Y_n, \cdot ) \right |d \td \eta_n\\
&\quad \quad + \int_{ \P\calE}  \left|   \Phi_M(Y_n, \cdot ) - \Phi_M(Y_\infty, \cdot )\right |d \td \eta_n\\
&\quad \quad +  \int_{ \P\calE}  \left|  \Phi(Y_\infty, \cdot ) - \Phi_M(Y_\infty, \cdot ) \right |d \td \eta_\infty\\
&\le 3\delta.
\end{align*}

Let $g_\infty= \exp (Y_\infty)$.  Note for each $n$ that  $$\int _X\log \|\calA(g_\infty^n, x)\| \ d \eta_\infty(x) \ge  \int _{\P \calE} \log(\left \|\calA(g_\infty^n, x) v \right \| \|v\|\inv) \ d\td  \eta_\infty (x,[v]).$$ It then follows   for any $\delta>0$ 
\begin{align*}
 \lambda_{\top, g_\infty, \eta,\calA }
  & = \lim _{n\to \infty} \frac 1 n  \int_X \log \|\calA(g_\infty^n, x)\| \ d \eta_\infty (x)\\
 &\ge  \liminf _{n\to \infty}  \frac 1 n \int _{\P \calE} \log\left( \left \|\calA(g_\infty^n, x) v \right \| \|v\|\inv \right) \ d\td  \eta_\infty (x,[v])\\
 &=  \liminf _{n\to \infty} \frac 1 n   \int _{\P \calE} \Phi (n Y_\infty , (x,[v]))   \ d\td  \eta_\infty (x,[v])\\
 &=    \int _{\P \calE} \Phi ( Y_\infty , (x,[v]))   \ d\td  \eta_\infty (x,[v])\\
 &\ge   \liminf _{n\to \infty} \int_{ \P\calE}   \Phi(Y_n, \cdot ) \ d \td \eta_n - 3\delta.
\end{align*}
where   the third equality follows from the invariance of $\td \eta_\infty$ and the cocycle property of $\Phi$.  
Since 
$$\liminf_{n\to \infty} \int_{ \P\calE}   \Phi(Y_n, \cdot ) \ d \td \eta_n \ge \epsilon $$
we conclude that \[\lambda_{\top, g_\infty,\eta,\calA}\ge\epsilon-3\delta\]
for any $\delta>0$ whence the result follows.
\end{proof}

\subsection{Oseledec's theorem for cocycles over actions by higher-rank abelian groups}
Let $A\subset G$ be a split Cartan subgroup.  Then $A\simeq \R^d$ where $d$ is the rank of $G$.
We have the following consequence of the higher-rank Oseledec's multiplicative ergodic theorem  (c.f.\ \cite[Theorem 2.4]{AWB-GLY-P1}).

Fix any norm $| \cdot |$ on $A\simeq \R^d$ and let $\eta\colon X\to \R$ be $$\eta(x) := \sup_{|a|\le 1} \log \| \calA(a, x)\|.$$
\begin{proposition}
\label{thm:higherrankMET}Let $\mu$ be an ergodic, $A$-invariant Borel probability measure on $X$ and
suppose $\eta\in L^{d,1}(\mu)$.  Then there are
	\begin{enumerate}
	\item an $\alpha$-invariant subset $\Lambda_0\subset X$ with $\mu(\Lambda_0)=1$;
  \item 
   linear functionals $\lambda_i\colon A \to \R$ for $1\le i\le p$;  
	\item   and splittings   $\calE(x)= \bigoplus _{i=1}^p E_{\lambda_i}(x)$ 
	into families of mutually transverse,  $\mu$-measurable  subbundles $E_{\lambda_i}(x)\subset \calE(x)$ defined  for $x\in \Lambda_0$

	\end{enumerate}
such that
\begin{enumlemma}	
	\item $\calA (s, x) E_{\lambda_i}(x)= E_{\lambda_i}(s\cdot x)$ and
	\item \label{lemma:partb} $\displaystyle \lim_{|s|\to \infty} \frac { \log \|  \calA (s,x) (v)\| - \lambda_i(s)}{|s|}=0$
\end{enumlemma}	
	for all $x\in \Lambda_0$ and all $ v\in  E_{\lambda_i}(p)\sm \{0\}$.  
 \end{proposition}
Note that \ref{lemma:partb} implies for $v\in E_{\lambda_i}(x)$ the weaker result that for $s\in A$,
$$\lim_{k\to\pm \infty} \tfrac {1} k \log \|  \calA (s^k,x) (v)\| =  \lambda_i(s).$$
Also note that for $s\in A$, and $\mu$ an $A$-invariant, $A$-ergodic measure that
\begin{equation}\label{eq:lameducksoup}\lambda_{\top, s, \mu,\calA} = \max _i \lambda_i(s).\end{equation}
If $\mu$ is not $A$-ergodic, we have the following.
\begin{claim} \label{claim:defenestratethepresident}
Let $\mu$ be an $A$-invariant measure with $\eta\in L^{d,1}(\mu)$ and $\lambda_{\top, s, \mu,\calA}>0$ for some $s\in A$. Then there is an $A$-ergodic component $\mu'$ of $\mu$ with
\begin{enumerate}
\item $\eta\in L^{d,1}(\mu')$;
\item there is non-zero   Lyapunov exponent $\lambda_j\neq 0$ for the $A$-action on $(X, \mu').$
\end{enumerate}
\end{claim}

We have the following which follows from the above definitions.
\begin{lemma}\label{lem:huntthelameduck}
Let $\mu$ be an $A$-invariant  probability  measure on $X$ with exponentially small mass in the cusps.
Suppose that $\calA$ is a tempered cocycle.
Then $\eta\in L^q(\mu)$ for all $q\ge 1$.   In particular, $\eta\in L^{d,1}(\mu)$.
\end{lemma}
%

\subsection{Applications to the suspension action}
\label{subsection:repackagingpreliminaries}

We summarize the  previous  discussion in the setting in which we will apply the above results in the sequel.   Recall we work with in a    fiber bundle with compact fiber $$M \rightarrow M^{\alpha}=(G \times M)/\Gamma \xrightarrow{\pi} G/\Gamma$$
 over non-compact
base $G/\Gamma$.
From  the discussion in  \cite[Section  2.1]{AWBFRHZW-latticemeasure}, we may equip $G\times M$ with a $C^1$ metric that is \begin{enumerate}
\item  $\Gamma$-invariant;
\item the restriction to $G$-orbits coincides with the fixed right-invariant metric on $G$;
\item there is a Siegel fundamental set $D\subset G$ on which the restrictions to the fibers of the metrics are uniformly comparable.
\end{enumerate}
The metric then descends to a $C^1$ Riemannian metric on $M^\alpha$.  We fix this metric for the remainder.  It follows that the diameter of any fiber of $M^\alpha$ is uniformly bounded.
It then follows that if $\mu$ is a measure on $M^\alpha$ then the image $\nu= \pi_* \mu$ in $G/\Gamma$ has  exponentially small mass in the cusps if and only if $\mu$ does; {\blue moreover, a family $\{\mu_\zeta\}$ of probability measures on $M^\alpha$ has  uniformly exponentially small mass in the cusps if and only if the family of projected  measures  $\{\pi_*\mu_\zeta\}$ on $G/\Gamma$ does.}
 Note that by averaging the metric over the left-action of $K$, we may also assume that the metric is left-$K$-invariant.  This, in particular, implies the right-invariant metric on $G$ in $(2)$ above is left-$K$-invariant.

{\blue For the remainder, the cocycle of interest will be the fiberwise derivative cocycle on the fiberwise tangent bundle, $$\calA(g,x)\colon F\to F,\quad \calA(g,x)= \restrict{D_xg}{F}.$$
Given  $g\in G$ and a $g$-invariant probability measure on $M^\alpha$,
the average leading   Lyapunov exponent for the fiberwise derivative cocycle for translation by $g$  is written either as $\lambda^F_{\top,\mu,g}$
or as  $\lambda_{\top,\mu,g, \calA}$.}



%
%

The next observation we need is a variant of a fairly standard observation about cocycle
over the suspension action.

\begin{lemma}
\label{lemma:tempered}
The fiberwise  derivative cocycle $\restrict{D_x g}{F}$ is tempered.
\end{lemma}

\begin{proof}
{\blue Write $\pi\colon M^\alpha \to G/\Gamma$.  By the construction of the metric in the fibers of $M^\alpha$ there is a $C>0$ with the following properties: given  $x\in M^\alpha$ and $g\in G$, writing $\bar x = \pi (x) \in G/\Gamma$ we have $$\| \restrict{D_x g}{F}\| \le C^{\beta(g,\bar x)+1}$$
and
$$m(\restrict{D_x g}{F})\ge C^{-\beta(g,\bar x)-1}.$$}
The conclusion is then an immediate consequence of Lemma \ref{lemma:fromlmr}. 
\end{proof}

We now assemble the consequences of the results in this section in the form we will use them below in a pair of lemmas.
The first is just a special case of Corollary \ref{corollary:finite}.

\begin{lemma}
\label{lemma:finite}
Let $s\in A$ and let $\nu$ be an $s$-invariant measure on $G/\Gamma$ with exponentially small mass in the cusps.  Let $\mu$ be an $s$-invariant measure on $M^\alpha$ projecting to $\nu$.  Then the average leading   Lyapunov exponent for the fiberwise derivative cocycle, $\lambda^F_{\top,\mu,s},$ is finite.
\end{lemma}

The second lemma summarizes the above abstract results in the setting of $G$ acting on $M^\alpha$.  

\begin{lemma}
\label{lemma:averagingsuspension}
Let $s\in A$ and let $\nu$ be an $s$-invariant measure on $G/\Gamma$ with exponentially small mass in the cusps.  
 Let $\mu$ be an $s$-invariant measure on $M^\alpha$ projecting to $\nu$.
\begin{enumerate}
\item \label{pla}For any amenable subgroup $H\subset C_G(s)$, if $\nu$ is $H$-invariant then
\begin{enumlemma}
\item  \label{playa} for any   \Folner sequence  of precompact sets $F_n$ in $H$,
the family $\{F_n \ast \mu\}$  has uniformly exponentially small mass in the cusps; and \item \label{playb} for any subsequential limit
$\mu'$ of $\{F_n \ast \mu\}$ we have
  $$\lambda_{\top, s, \mu}^F \le \lambda_{\top, s, \mu'}^F.$$ 
\end{enumlemma}
\item  
For any one-parameter unipotent subgroup $U$ centralized by $s$
\begin{enumerate}
\item the family
 $\{U^T \ast \mu\}$ has uniformly exponentially small mass in the cusps;  and \item for any accumulation point   $\mu'$ of $\{U^T \ast \mu\}$  as $T\to \infty$ we have
$$\lambda_{\top, s, \mu} ^F\le \lambda_{\top,s, \mu'}^F.$$
\end{enumerate}
\end{enumerate}
\end{lemma}

\begin{proof}
Part \ref{playa} of the first conclusion is immediate since $H$-invariance of $\nu$ implies $\nu=\pi_*(F_n\ast \mu)$ for all $n$;  
part \ref{playb} then follows from 
Lemma \ref{lemma:averaging}.
The second conclusion follows from  Proposition \ref{prop:bananas} 
and Lemma \ref{lemma:averaging}.  
\end{proof}

We remark that we will also use Lemma \ref{lemma:firstexponents} in the proof of the main theorem, but we do not reformulate a special case of it here since the reformulation adds little clarity.


\section{Subexponential growth of derivatives for unipotent elements}
\label{unipotents}

In this section we show that the restriction of the action $\alpha$ to certain unipotent elements in each copy  $\Lambda_{i,j} \cong \Sl(2, \Z)$ have  uniform  subexponential growth of   derivatives with respect to a right-invariant distance on $\Sl(2, \R)$.  Note that each $\Sl(2,\R)$ is geodesically embedded whence the   $\Sl(2,\R)$ distance is the same as  the $\Sl(m, \R)$ distance.  By  \cite{MR1244421, MR1828742},  the $\Sl(m,\R)$ distance is  quasi-isometric to the word-length in $\Sl(m,\Z)$.   Recall   that $d(\cdot,\cdot)$ denotes a right-invariant distance  on $\Sl(m, \R)$ and that $\Id$ is the identity in $\Sl(m,\R)$.

For $ 1\leq i < j \neq n$, let $\Lambda_{i,j} \cong \Sl(2, \Z)$ be the copy of $\Sl(2, \Z)$ in $\Sl(m,\Z)$ corresponding to the elements in $\Sl(m,\Z)$ which acts only on the lattice $\Z^2 < \Z^m$ generated by $\{e_i,e_j\}$.
Note that as all $\Lambda_{i,j}$ are conjugate under the Weyl group, it suffices to work with one of them.

Define the unipotent element $u := \begin{bmatrix} 1 & 1 \\ 0 & 1 \end{bmatrix}$ viewed as an element of $\Lambda_{i,j}$.  Note that any upper or  lower triangular unipotent element of $\Lambda_{i,j}$ is  conjugate to a power of $u$  under the Weyl group.

\begin{proposition}[Subexponential growth of derivatives for unipotent elements]
\label{unipotentisgood} For any $\Lambda_{i,j}$ and any $\e > 0$, there exists $N_{\e}>0$ such that for any $n \geq N_\e$: $$\|D(\alpha(u^n))\| \leq e^{\e d(u^n,\Id)}$$
\end{proposition}

To establish Proposition \ref{unipotentisgood}, we first show that generic elements in $\Sl(2, \Z)$ have uniform subexponential growth of derivatives.  This first part requires reusing most of the key arguments from \cite{BFH} in a slightly modified form.  We encourage the reader to read that paper first.

\subsection{Slow growth for ``most" elements in $\Sl(2,\Z)$}
\label{subsection:slowgrowthgeneric}
For $\epsilon>0$, $k >0$, and $x \in \Sl(2, \R)$, we make the following definitions:

\begin{enumerate}
\item For $S\subset \Sl(2,\R)$  let $|S|$ denote the Haar-volume of $S$.
\item Let $K = \So(2) \subset \Sl(2,\R)$.  For $S\subset  K$   let $|S|$ denote the Haar-volume of $S$.
\item Let $B_k(x)$ denote the ball of radius $k$ centered at $x$ in $\Sl(2, \R)$.
\item Let $T_k := B_k(\Id) \cap \Sl(2, \Z)$.  Given $S\subset \Sl(2,\Z)$ write $|S|$ for the cardinality of $S$.
\item Define the  set of $\e$-bad elements to be $$M_{\e, k} := \{ \gamma \in T_k \text{ such that }   \|D(\alpha(\gamma))\| \geq e^{\e k }\}.$$
\item Define the set of $\e$-good elements to be $$G_{\e,k} := T_k \setminus M_{\e, k} .$$

\end{enumerate}

To establish Proposition \ref{unipotentisgood}, we first  show that the set $G_{\e,k}$ contains a positive proportion of $T_k$ when $k$ is large enough.
\begin{proposition}\label{mainunipo} For any $\delta > 0$,  the set $G_{\e,k}$ has at least $(1- \delta)|T_k|$ elements for every sufficiently large $k$. 
\end{proposition}

We have the following  well-known fact.  {\blue See for instance \cite[Section 2]{MR1230290}.}

\begin{lemma}\label{basic} There exist  positive constants $c, C$ such that for any $k \geq0$: $$ c|B_k| \leq |T_k| \leq C|B_k|.$$
\end{lemma}

For an element $x \in \Sl(2,\R)$, let $\bar{x}$ denote the projection in $\Sl(2,\R)/\Sl(2,\Z)$. Define
\def\fib{\text{Fiber}}
$$\|D_{\bar{x}} g\|_{\text{Fiber}} = \sup\{ \|\restrict{D_y g}{F}\| :y\in M^\alpha,  \pi(y) = \bar{x}\} .$$
Let $$G'_{\e,k}(x) := \{ g \in B_k( x) \text{ such that } \|D_{\bar{x}}g\|_{\text{Fiber}} \leq e^{\e d(g,\Id)}\}.$$


\begin{lemma}\label{uni1} For almost every  $x \in \Sl(2,\R)$ and  any $\delta > 0$ we have   $$|G'_{\e, k}(x)| > (1- \delta) |B_k|$$ for all    $k$ sufficiently  large.
\end{lemma}\label{lem:aegood}
\def\calM{\mathcal M}
\begin{proof}
 Let $a^t\in \Sl(2,\R) $ be the matrix $$a^t :=  \begin{bmatrix} e^t&0\\0&e^{-t} \end{bmatrix}.$$ Recall that  the action of the one-parameter diagonal subgroup $\{a^t\}$  on $ \Sl(2,\R)/\Sl(2,\Z)$ is ergodic with respect to Haar measure.

{\blue Let $\calM$ denote the set of Borel probability measures on $\Sl(2,\R)/\Sl(2,\Z)$ equipped with the standard topology (dual to bounded  continuous functions).  The topology on $\calM$ is metrizable (see \cite[Theorem 6.8]{MR1700749});  fix a metric on $\rho_\calM$ on $\calM$. } 

Consider the function $\psi\colon \Sl(2,\R)/\Sl(2,\Z)\to \R$ given by  $\psi(x) := e^{\eta d(x,x_0)}$ where $x_0=   \Sl(2,\Z)$ is the identity coset and $\eta>0$ is chosen sufficiently small so that $\psi$ is $L^1$ with respect to the Haar measure.
 By the pointwise ergodic theorem, for almost every  $x \in \Sl(2,\R)$ and almost every  $k_1 \in \SO(2)$ we have
\begin{equation}\label{eq:assin8thepresident} \lim_{T \to \infty} \frac{1}{T} \int_{0}^T \psi(a^tk_1\bar{x}) \ dt = \int_{\Sl(2,\R)/\Sl(2,\Z)} \psi \ d\text{Haar} < \infty.\end{equation}
Similarly, for almost every  $x \in \Sl(2,\R)$ and almost every  $k_1 \in \SO(2)$ we have  \begin{equation}\label{eq:grabtrumpbythepussy}\lim_{T \to \infty}  \frac{1}{T} \int_{0}^T \delta_{a^tk_1\bar{x}}  \ dt =  \text {Haar}.\end{equation}
Let $S\subset \Sl(2,\R)$ be the set of  $x\in \Sl(2,\R)$ such that \eqref{eq:assin8thepresident} and  \eqref{eq:grabtrumpbythepussy} hold for almost every $k_1 \in \So(2)$.  The set $S$ is $\Sl(2,\Z)$-invariant and co-null.  We show any $x\in S$ satisfies the conclusion of the lemma.

For fixed $x\in S$ and fixed $\delta>0$,  there exist   $T_{\delta}= T_\delta(x)$, a sequence $T_j = T_j (x)$ for $j\in \N$,  and a set $K_{\delta} = K_\delta(x) \subset \So(2)$ such that  $|K_{\delta}| \geq (1-\delta/2)|\So(2)|$ with the property that  for any $k_1 \in K_ {\delta}$ and any $T \geq T_{\delta}$ we have

\begin{equation}\label{tight}
\frac{1}{T} \int_{0}^T \psi(a^tk_1\bar{x}) \ dt < 2\int_{\Sl(2,\R)/\Sl(2,\Z)} \psi \ d\text{Haar}
\end{equation}
and for each $1\le j$
\begin{equation}\label{tight2}
\rho_\calM \left(\frac{1}{T} \int_{0}^{T}  \delta_{a^tk_1\bar{x}}  \ dt, \text{Haar}\right)< \frac{1}{j}.
\end{equation}
\noindent for all $T \geq T_j$.
 To finish the proof of the lemma, define the set $$G''_{k}(x) := \{   k_1a^tk_2 \text{ where } k_1 \in \So(2), k_2 \in K_{\delta}(x) \text{ and } (\delta/2) k <t < k\}.$$
 For $k$ large enough, we have that  $|G''_{k}(x)| \geq (1- \delta)|B_k|$.
   We claim   that \begin{equation}\label{eq:okeqn}G''_{k}(x) \subset G'_{\e,k}(x)\end{equation} for $k$ sufficiently large.     For the sake of contradiction, suppose \eqref{eq:okeqn} fails.  Using that the norm on $F$ is chosen to be $K$-invariant, there exists  $x_n \in \Sl(2,\R)$ with each $x_n $ in the  $K_\delta(x)$-orbit of $x$ such that $\|D_{x_n}(a^{t_n})\|_\fib \geq e^{\e t_n}$ for some sequence $t_n \to \infty$.  Moreover, the corresponding {empirical measures} $$\eta_n:=\frac{1}{t_n} \int_{0}^{t_n}  \delta_{a^t\bar{x}_n}  \ dt$$ have uniformly exponentially small mass in the cusps by equation \eqref{tight}.

By Lemma \ref{lemma:firstexponents} and  \eqref{tight2}, a subsequence of the measures $\eta_n$ converge to an $a^t$-invariant measure $\mu_0$ on $M^{\alpha}$  whose projection to $\Sl(m, \R)/\Sl(m,\Z)$ is Haar measure on the embedded modular surface $\Sl(2,\R)/\Sl(2, \Z)$ and has positive fiberwise Lyapunov exponent for the action of $a^1$. 
Since $a^t$ is ergodic on   $\Sl(2,\R)/\Sl(2, \Z)$, we can assume $\mu_0$ is ergodic by taking an ergodic component without changing any other properties.

We average as in \cite{BFH} to improve $\mu_0$ to a measure whose projection is the Haar measure on  $\Sl(m, \R)/\Sl(m,\Z)$.  Difficulties related to escape of mass are   handled by the   preliminaries in Section \ref{section:preliminaries}.

 As above, we note that there is a canonical copy of $H_2 = \Sl(m-2, \R)$ in $\Sl(m,\R)$ commuting with our chosen $H_1= \Sl(2,\R)$. {\blue Recall $A$ is the Cartan subgroup of $\Sl(m,\R)$ of positive diagonal matrices.  
 The subgroup $A$ contains the one-parameter  group $\{a^t\}$ and a Cartan subgroup of $H_2$.} Let
 \begin{itemize}\item $A_1 = A \cap H_1 = \{a^t\}$,\item  $A_2 = A \cap H_2$, and \item $A' = A \cap H_1 \times H_2$. \end{itemize}Note that $A' <A$ has codimension one.  Our chosen modular surface $\Sl(2,\R)/\Sl(2,\Z) \subset \Sl(m,\R)/\Sl(m,\Z)$  is such that \begin{align*}\Sl(2,\R)/\Sl(2,\Z) &\subset \Sl(2,\R)/\Sl(2,\Z) \times \Sl(m-2 ,\R)/\Sl(m-2,\Z)\\& \subset \Sl(m,\R)/\Sl(m,\Z).\end{align*}

 {\blue Define an $A'$-ergodic,  $A'$-invariant  measure $\mu_1$ on $M^{\alpha}$ that projects to Haar measure on $\Sl(2,\R)/\Sl(2,\Z) \times \Sl(m-2 ,\R)/\Sl(m-2,\Z)$ as follows: Let $M^{\alpha}_{2,m-2}$ denote the  restriction of the fiber-bundle  $M^{\alpha}$ to $\Sl(2, \R)/\Sl(2,\Z) \times \Sl(m-2,\R)/\Sl(m-2, \Z)$.  Pick point $y$ in $\Sl(m-2,\R)/\Sl(m-2, \Z)$ that equidistributes to the Haar measure on $\Sl(m-2,\R)/\Sl(m-2, \Z)$ under a \Folner sequence in $A_2$.  Consider $\mu_0$ as a measure on the restriction of $M^{\alpha}$ to $\Sl(2, \R)/\Sl(2,\Z) \times \{y\}$.  Now average $\mu_0$ over a \Folner   sequence  in $A_2$ and take a limit $\hat \mu_1$.  Note that  $\hat \mu_1$ has positive fiberwise Lyapunov exponent  $ \lambda_{\top, a^1,\mu_1}^F>0$. This can be seen by  mimicking the proof of   Lemma \ref{lemma:averaging}. Let $\mu_1$ be an $A'$ ergodic component of $\hat \mu_1$, then the measure $\mu_1$ has the desired properties and is  supported on the subset of $M^{\alpha}$ defined by restricting the bundle to $\Sl(2,\R)/\Sl(2,\Z) \times \Sl(m-2 ,\R)/\Sl(m-2,\Z)$.
}

We consider the $A'$-action on $(M^\alpha, \mu_1)$ and the fiberwise derivative cocycle $\calA(g,y) = \restrict{D_y g}{F}$.
By \eqref{eq:lameducksoup}, there is a   non-zero  Lyapunov exponent
$\lambda_{ \mu_1,A'}^F\colon A'\to \R$ for this action.   We apply the averaging procedure in   Proposition \ref{proposition:averaging} to this measure.
    Take  $\beta'$ to be either $\alpha_2$ or $\delta$ so that $\beta'\colon A'\to \R$ is not proportional to $\lambda_{ \mu_1,A'}^F$.
Choose $a_0\in A'$ such that $a_0 \in \ker(\beta')$ and $\lambda_{ \mu_1,A'}^F(a_0)>0$.
Let $U = U^{\beta'}$ and let $\mu_2$ be any subsequential limit  of $U^T\ast \mu_1$ as $T \rightarrow \infty$. Then
 $\mu_2$  is  $a_0$-invariant, and has positive fiberwise Lyapunov exponent $ \lambda_{\top, a_0, \mu_2}^F >0$.  Moreover,
 $\pi_*\mu_2$ is $H$-invariant.
By Lemma \ref{lemma:averagingsuspension}  and Proposition \ref{proposition:averaging}, $\mu_2$ has exponentially small mass in the cusps.
We may also assume $\mu_2$ is ergodic by passing to an ergodic component and by Claim \ref{claim:defenestratethepresident} assume $\mu_2$ has a non-zero fiberwise  Lyapunov exponent $\lambda_{ \mu_2,A'}^F$ for the $A'$-action.

 We now average $\mu_2$ over $A'$ to obtain $\mu_3$.   Then $\mu_3$ has a non-zero fiberwise  Lyapunov exponent $\lambda_{\mu_3,A'}^F$ and has  exponentially small mass in the cusps by   Lemma \ref{lemma:averagingsuspension}(\ref{pla}).   Since $\pi_*\mu_2$ was $A'$-invariant, we have  $\pi_*\mu_2= \pi_*\mu_3$.  Once again, we may pass to an $A'$-ergodic component of $\mu_3$ that retains the desired properties.

Take $\hat \beta$ to be either $-\alpha_2$ or  $-\delta$ so that $\hat \beta$ is not proportional to $\lambda_{\mu_3,A'}^F$ on $A'$.
Select $a_1$ with $\lambda_{\mu_3,A'}^F(a_1)>0$ and  $\hat \beta(a_1)=0$.  By Proposition \ref{proposition:averaging} and Lemma \ref{lemma:averagingsuspension}, we obtain a new measure $\mu_4$ with  $\pi_*\mu_4$ the Haar measure on $\Sl(m, \R)/\Sl(m,Z)$.  We have $ \lambda_{\top, a_1,\mu_4}^F >0$.
Finally, average $\mu_4$ over all of $A$ to obtain $\mu_5$.  Since $\pi_*\mu_4$ is the Haar measure and thus   $A$-invariant, we have that $\pi_*\mu_4= \pi_*\mu_5$.
By Lemma \ref{lemma:averagingsuspension},      $\mu_5$ has a non-zero fiberwise Lyapunov exponent $\lambda_{\mu_5,A}^F$ for the action of $A$.  Replace  $\mu_5$ by an ergodic component with positive fiberwise Lyapunov exponent.

Exactly as in \cite[Section 5.5]{BFH}, we apply \cite[Proposition 5.1]{AWBFRHZW-latticemeasure} and conclude that  $\mu_5$ is a $G$-invariant measure on $M^{\alpha}$.  We then obtain a contradiction with  Zimmer's cocycle superrigidity theorem.
To conclude that $\mu_5$ is a $G$-invariant, note that \cite[Proposition 5.1]{AWBFRHZW-latticemeasure} holds for actions induced from actions of any lattice in $\Sl(m,\R)$ and shows that   $\mu_5$ is invariant under   root subgroups corresponding to \emph{non-resonant} roots.   Dimension counting exactly as in   \cite[Section 5.5]{BFH}  shows that the  non-resonant roots of $\Sl(m,\R)$  generate all of $G$ if  the  dimension of $M$ is at most $m-2$ or if the dimension of $M$ is $m-1$ and  the action is  preserves a  volume.
\end{proof}

We derive Proposition \ref{mainunipo} from Lemma \ref{uni1}.
\begin{proof}[Proof of Proposition \ref{mainunipo}]
Fix $0<c<1$ sufficiently small so that if $d(\Id,g)<c$ then $\|D_{\Gamma} g\|_{\text{Fiber}}\le e^{\epsilon/4}$.
Fix a point $x\in \Sl(2,\R)$  as in   Lemma \ref{uni1} with $d(\Id, x)<c$.
Observe that if $k\ge 1$ and $g \in G'_{ \epsilon/4,k}(x)$, then $gx \in G'_{\e/2,k+c}(\Id)$.
In particular, for any $\delta>0$  we have  for all $k$ sufficiently large that 
\begin{equation}\label{eq:sendninjastodecapitateTrumpinhissleep} |B_{k+c}\sm G'_{\e/2, k+c}(\Id)| < \delta \hat C |B_k|\end{equation} where $\hat C$ is a constant depending on $c$.

Take  $U$ to be the ball of radius $c$ centered at  the identity coset in $\Sl(2,\R)/\Sl(2,\Z)$ and consider   lifts of $U$ to $\Sl(2,\R)$ intersecting the ball $B_k$. If a lift of $U$ intersects $G'_{\e/2, k+ c}(\Id)$, then the corresponding element of the deck group $\Sl(2, \Z)$ belongs to $G'_{3\e/4,k}(\id)$.

{\blue Let $\tilde U$ be the set of lifts of $U$.  From Lemma  \ref{basic} and \eqref{eq:sendninjastodecapitateTrumpinhissleep}, it follows
that ratio of the measure of $\tilde{U} \cap B_k \cap  G'_{\e/2, k}(\Id)$ to the measure of $\tilde{U} \cap B_k$ goes to one as $k \to \infty$.}
Finally, since the norms on the fiber of $M^\alpha$ above the identity coset and the original norm on $M$ are uniformly comparable,  the   result follows.
\end{proof}

\begin{remark} Using large deviations, one can   make $\delta$ to be decreasing with $k$,  roughly as $\delta_{k} = e^{-k^{1/1000}} $.   See \cite{MR2247652, MR2787598}.  This is not necessary for our argument.
\end{remark}

\subsection{Subexponential growth of derivatives for unipotent elements in $\Sl(2,\Z)$}\label{sec:mutualmastication}

We work here with a specific copy of the group $\Sl(2,\R)\ltimes\R^2$ embedded in $\Sl(m,\R)$ and its intersection with the lattice $\Gamma$; the copy of $\Sl(2,\R)\ltimes\R^2$ corresponds to the elements of $\SL_m(\R)$ which differ from the identity matrix only in  the first two rows and  first three columns.
Any   unipotent element of   any $\Lambda_{i,j}\subset \Gamma$ considered in   the statement of Proposition \ref{unipotentisgood} is conjugate by an element of the  Weyl group to a power of the elementary matrix $E_{1,3}$.  Thus, after conjugation, any such element is contained in the  distinguished     copy of  $\Sl(2,\Z) \ltimes \Z^2$ generated by  $\Sl(2,\Z) = \Lambda_{1,2}$ and the normal subgroup $\Z^2$   generated by $E_{1,3}$ and $E_{2,3}$.

 For the reminder of this subsection, we   work with this fixed group. Identify $H_{1,2}$ with $\Sl(2,\R)$. Let $U_{1,2}:=  \{ u_{a,b}\}$ denote the abelian subgroup of $\Sl(m,\R)$ consisting of unipotent elements of the form 
$$
u_{a,b}:= \left(\begin{array}{ccccc}1  &  0  & a  &  &   \\   0 & 1  &   b  & &   \\   0 & 0 & 1 &   &   \\  &   &   &  \ddots &  \\  &   &   &   &1  \end{array}\right)$$
Clearly, $U_{1,2}$ is normalized by $H_{1,2}$ and  $H_{1,2} \ltimes U_{1,2}  \cong \Sl(2,\R) \ltimes \R^2$.    We have an embedding $$ \Sl(2,\R) \ltimes \R^2 / \Sl(2,\Z)\ltimes \Z^2 \to \Sl(m,\R)/\Sl(m,\Z)$$ where $\Z^2$ is identified with the subgroup generated by the unipotent elements $u_{1,0}$ and $u_{0,1}$. Note that $\Sl(2,\R) \ltimes \R^2 / \Sl(2,\Z)\ \ltimes \Z^2$ is a torus bundle over the unit-tangent bundle of the modular surface.

Equip $\Z^2$ with the $L_\infty$ norm with respect to the generating set $\{u_{1,0}, u_{0,1}\}$ and let $B_n(\Z^2)$ denote the closed ball of radius $n$ in $\Z^2$ centered at $0$ with respect to this norm.  Given $S\subset \Z^2$ let   $|S|$ denote the cardinality of the set  $S$.

Define the set of ``$\epsilon$-good unipotent  elements'' of $\Z^2\subset \Gamma$, denoted  by $GU_{\e,n}$, to be the following subset of $\Z^2$:
\begin{equation}\label{eq:racistPOTUS} GU_{\e,n} := \left\{ u_{a,b} \in B_n (\Z^2) \text{ such that } \|D(\alpha(u_{a,b}^{\pm1}))\| \leq e^{\e \log(n)}\right\}.\end{equation}

The main results of this subsection is the following.
\begin{proposition}\label{finalunipotent} For any $\e>0$, there exists $N_\e > 0$ such that if $n \geq N_{\e}$, then $GU_{\e,n} = B_n(\Z^2)$
\end{proposition}

Proposition \ref{unipotentisgood} follows from Proposition \ref{finalunipotent} using that any  subgroup $\langle u^n \rangle$ in Proposition \ref{unipotentisgood} is conjugate to a subgroup of the group $\Z^2$ and  the fact that  $d(u^n, \Id) = O(\log(n))$ from \eqref{unipotentgrowth}. 
The proof of Proposition \ref{finalunipotent} consists of conjugating elements of $U_{1,2}$ by elements of $G_{\e,n}$ in order to obtain a subset of $G_{\e,n}$ that contains  a positive density  of elements of $B_n(\Z^2)$.  Then, using the fact that $\Z^2$ is abelian, we  promote such a subset to all of $B_n(\Z^2)$  by taking sufficiently large sumsets  in Proposition \ref{babycombinatorics}.

\begin{lemma}\label{proportion}
There exists $\delta' > 0$ with the following properties: for any   $\e>0$ there is an $N_{\e}'>0$ such that for any $n \geq N_{\e}'$ we have  $$|GU_{\e,n}| \geq \delta' |B_n(\Z^2)|.$$
 \end{lemma}

\begin{proof}
Recall that $T_k$ denotes the intersection of the ball of radius $k$ in $\Sl(2,\R)\simeq H_{1,2}$ with $\Sl(2,\Z)= \Lambda_{1,2}$ and $|T_k|$ denotes the cardinality of $T_k$.  As $|T_k|$ grows exponentially in $k$, we may take
 $s$ fixed  so  that $|T_{k-s}| < \frac{1}{2}|T_{k}|$ for all $k$ sufficiently large. Given $\epsilon '>0$, define the subset $S_k\subset \Sl(2,\Z)$ to be   $$S_k := G_{\e',k}\cap G_{\e',k}\inv \cap  (T_k \setminus T_{k-s}).$$ From Proposition \ref{mainunipo}, we may  assume that $$|S_{k}| \geq \frac{1}{2} |T_k|.$$

From   \eqref{normdistance}, there exists $C_1>0$ such that if $A = \begin{bmatrix} a&c\\b&d \end{bmatrix}$ belongs to $S_k$ then either $$\text{$\|(a,b)\|_\infty \geq C_1e^{ \frac{1}{2}{(k-s)} }$ or $\|(c,d)\|_\infty \geq C_1e^{\frac{1}{2}{(k-s)} }$.}$$ Without loss of generality,   we     assume that at least half of the elements in $S_k$ satisfy   $\|(a,b)\|_\infty \geq C_1e^{ \frac{1}{2}{(k-s)} }$.

Consider the map $P\colon S_{k} \to \Z^2$ that assigns     $A = \begin{bmatrix} a&c\\b&d \end{bmatrix}$ to $(a,b)$. By  \eqref{normdistance},  there is $C_2>1$ such that  the image $P(S_k)$ of $S_k$ lies in the norm-ball $B_{C_2e^{\frac{k}{2}}}(\Z^2)$ for all $k$.

{\blue
 Let $k(n) = 2\log (n)-\log C_2 $. Then $P(S_{k(n)}) \subset B_{n}(\Z^2)$.  If  $n$ is sufficiently large and $A = \begin{bmatrix} a&c\\b&d \end{bmatrix} \in S_{k(n)}$ then  we have $u_{a,b} \in GU_{(5\e',n)}$; indeed $$\alpha (u_{a,b})  =\alpha (A) \circ \alpha (u_{1,0})\circ  \alpha (A\inv)$$
whence $$\|D \alpha (u_{a,b}) \| \le \|D\alpha (u_{1,0})\| e^{2 \epsilon' k(n)}. $$
We have $|B_{   n}(\Z^2)|\le D_1 n^2$ for some $D_1\ge 1$.  Also, from \eqref{decapitationOrImpeachment?} and Lemma \ref{basic} we have   $|S_{k(n)}| \geq \frac{1}{2}|T_{k(n)}| \ge \frac{1}{2} e^{k(n)} = 
\frac 1{D_2}n^2$ for some $D_2\ge 1$.
}

To to complete the proof, we show that the preimage $P^{-1} ((a,b))$ {\blue in $S_k$} of any $(a,b)\in \Z^2$ {\blue satisfying $\|(a,b)\|_\infty \geq C_1e^{ \frac{1}{2}{(k-s)} }$} has uniformly bounded cardinality  depending only on  $s$.
 Observe that if $A,A'\in \Sl(2,\Z)$  satisfy $P(A) = P(A')$, then $A' = AU$, where $U = \begin{bmatrix} 1&m\\0&1 \end{bmatrix}$ for some $m \in \Z$ and we have 
 $$A=   \begin{bmatrix} a& c\\b&  d \end{bmatrix},\quad \text{and} \quad A' =   \begin{bmatrix} a&am+ c\\b&bm + d \end{bmatrix}.$$
If $A'$ belongs to $T_k$ then $\|(am +c, bm+d)\|_\infty \leq C_2e^{\frac{k}{2}}$ and if $A$ belongs to $T_k$ then $\|(c, d)\|_\infty \leq C_2e^{\frac{k}{2}}$.
We thus have   that $|am|\le 2 C_2 e^{{\blue \frac{k}{2}}}$ and  $|bm|\le 2 C_2 e^{{\blue \frac{k}{2}}}$.
As we assume that $$\|(a,b)\|_\infty \geq C_1e^{ \frac{k-s}{2} }$$ we have that  $|m|\le  2\frac{C_2}{C_1}e^{{{\blue \frac{s}{2}}}}$.  Thus, the preimage $P^{-1}((a,b))$  has  at most $4\frac{C_2}{C_1}e^{{{\blue \frac{s}{2}}}} +1$ elements {\blue in $S_k$} .

With $\epsilon'=\frac 1 5 \epsilon$, having taken $n$ sufficiently large, we thus have
{\blue \begin{align*}\frac{|GU_{\e,n}|}{|B_n(\Z^2)|} \ge \frac {1} {D_1n^2} {\frac{\frac{1}{2}|S_{k(n)} |}{4\frac{C_2}{C_1}e^{s/2} +1  }}\ge
 \frac{1}{2} \frac{ \frac{1}{D_2}n^2} {4\frac{C_2}{C_1}e^{s/2} +1}  \frac{1}{D_1n^2}=:\delta'
\end{align*}}
which completes the proof.\end{proof}


To complete the proof  of Proposition \ref{finalunipotent}, we show  that any element in $B_n(\Z^2)$ can be written as a product of a bounded number of elements in $GU_{\e,n}$ independent of $\e$.  This follows from the structure of sumsets of  abelian groups.  

From the chain rule and submultiplicativity of norms, we have the following.


\begin{claim}\label{trivialchain}
For any positive integers $n,m$ and $\e_1, \e_2 > 0$, if  $u_{a,b} \in GU_{\e_1,n}$ and $u_{c,d} \in GU_{\e_2,m}$   then the product $u_{a,b}u_{c,d} \in GU_{\max\{\e_1, \e_2\},n + m}$
\end{claim}


For subsets $A, B \subset \Z^2 $ we denote by     $A + B$ the sumset of $A, B$.   

\begin{claim}\label{babycombinatorics}
For any $0<\delta <1$, there exists a positive integer $k_{\delta}$ and a finite set $F_{\delta} \subset \Z^2$ such that for any $n$ and any symmetric set $S_n \subset B_n(\Z^2)$ 
with  $|S_n| >  \delta|B_n|$, we have that $$B_n\subset F_\delta + \underbrace{S_n + S_n + ... + S_n}_{\text{$k_{\delta}$ times}}.$$
\end{claim}

\begin{proof}
Fix $M\in \Z_+$ with $\frac 1 M<\delta$.
Take    $N_{\delta} := \big( M + 1 \big)!$, $k_\delta = 4 N_\delta$, and $F_{\delta} := B_{N_{\delta}}(\Z^2).$
Consider a symmetric set $S_n \subset B_n (\Z^2)$ with $|S_n| > \delta |B_n(\Z^2)|$.

If $n\le N_\delta$ then $B_n (\Z^2)\subset F_\delta$ and we are done.  Thus, consider $n\ge N_\delta$.  To complete the proof the claim, we argue that the set
$$\sum_{k_\delta} S_n := \underbrace{S_n + S_n + ... + S_n}_{\text{$k_{\delta}$ times}}$$ contains the intersection of  the sublattice $N_\delta \Z^2$ with $B_n(\Z^2)$.  Adding $F_\delta$ to the sumset then implies the claim.
  Consider any  non-zero vector
$\td v\in N_\delta \Z^2\cap B_n(\Z^2)$ of the form $(\td \ell, 0)$ for some $\td \ell \in [-n,n]\cap N_\delta \Z$.
Then $\td v = N_\delta v$ where $v= (\ell,0)$ is such that $0< |\ell |\le \lfloor n N_\delta\inv \rfloor$.

Consider the equivalence relation in $B_n(\Z^2)$ defined by declaring that two elements $x,y \in R(n)$ are equivalent if $x-y$ is an integer multiple of $v$.
Each equivalence class is of the form $$C_x = \{ ...,x - v, x, x + v, x + 2v,....\}.$$
As  $|S_n| \geq \frac{1}{M} |B_n(\Z^2)|$, there exists  one equivalence class $C_x$  such that $|C_x \cap S_n| \geq \frac{1}{M} |C_x|$.
Since $0< |\ell |\le \lfloor n N_\delta\inv \rfloor$, each equivalence class contains at least $M+1$ elements and hence $C_x \cap S_n$ contains at least two elements $a,b$ with $b= a + i v $ for $|i|\le M$.
In particular, since  $a-b = iv$, we have $iv \in S_n + S_n$.  As $i$ divides $N_{\delta}$, we have that $\td v = N_{\delta}v \in  \sum_{2N_{\delta}} S_n$.

Similarly, for $n\ge N_\delta$ and any  $\td u\in N_\delta \Z^2\cap B_n(\Z^2)$ of the form $(0, \td \ell)$  we have
$\td u \in  \sum_{2N_{\delta}} S_n$.  Then $$\td u  + \td v\in \sum_{4N_{\delta}} S_n$$
completing the proof. \end{proof}

\begin{proof}[Proof of Proposition \ref{finalunipotent}] Given $\epsilon'>0$,  let  $\delta'$ and $ N_{\e'}'$ be given by Lemma \ref{proportion}.  Let $S_n := GU_{\e',n}$ be as in \eqref{eq:racistPOTUS}  and take $k_\delta'$ and $F_{\delta'}$ as in  Lemma \ref{babycombinatorics}.  Note that $GU_{\e',n}$  is symmetric by definition.
Take $N\ge N'_{\e'}$  such  that $F_{\delta'} \in GU_{\e',n}$ whenever $n \geq N $. For $n\ge N$ and any $u_{a,b} \in B_n(\Z^2)$ we have that $u_{a,b} \in F_\delta + S_n + S_n+ ... + S_n $ ($k_{\delta'}$ times) by Proposition \ref{babycombinatorics}. Proposition \ref{trivialchain} then implies   that $u_{a,b} \in GU_{ \e', (k_{\delta'} +1)n}$ so $\|D(u_{a,b})^{\pm1}\| \leq e^{ \e' \log( (k_{\delta'} +1)n ) }$.  With  $\epsilon '= \epsilon /2$, take $N_\epsilon\ge \max\{N,(k_{\delta'} +1)\}. $ Then for all $n\ge N_\epsilon$ we have $$ \e' \log( (k_{\delta'} +1)n) \le  \e \log(n)$$
whence
$$\|D(u_{a,b})^{\pm1}\|\le e^{\epsilon \log(n)}$$
and for $u_{a,b} \in B_n(\Z^2)$ with $n\ge N_\epsilon$.
\end{proof}

\section{Proof of Theorem \ref{main2}}
\subsection{Reduction to the restriction of an action by $\Lambda_{i,j}$}
We recall the work of Lubotzky, Mozes, and Raghunathan, namely   \cite{MR1244421} and \cite{MR1828742}, which establishes quasi-isometry   between the word and Riemannian metrics on lattices in higher-rank semisimple Lie groups.
In the special case of $\Gamma= \Sl(m,\Z)$ for $m\ge 3$,
 in  \cite[Corollary 3]{MR1244421} it is shown that any element $\gamma$ of $\Sl(m,\Z)$ is written as a product of at most $m^2$ elements $\gamma_i$.  Moreover each $\gamma_i$ is contained some $\Lambda_{i,j}\simeq \Sl(2,\Z)$ and the word-length of each $\gamma_i$ is  proportional to the word-length of $\gamma$.

 Thus, to establish that an action $\alpha\colon \Gamma\to \Diff^1(M)$ has uniform subexponential growth of derivatives in  Theorem \ref{main2}, it  is sufficient to show that the restriction  $\restrict {\alpha}{\Lambda_{i,j}}\colon \Gamma\to \Diff^1(M)$ has uniform subexponential growth of derivatives for each $1\le i\neq j \le m$.  We emphasize that to measure subexponential growth of derivatives, the word-length on ${\Lambda_{i,j}}$ is measured as the word-length as embedded in $\Sl(m,\Z)$ (which is quasi-isometric to the Riemannian metric on $\Sl(m,\R)$) rather than the intrinsic word-length in $\Lambda_{i,j}\simeq \Sl(2,\Z)$ (which is not quasi-isometric to the Riemannian metric on $\Sl(2,\R)$).

As the Weyl group acts transitively on the set of all $\Lambda_{i,j}$, it is   sufficient to consider a fixed $\Lambda_{i,j}$.  Thus to deduce Theorem \ref{main2},  in the remainder of this section we establish the following, which is the  main proposition of the paper.
\begin{proposition}\label{prop:maybeweshouldstatethemainresultatsomepoint}
For any action $\alpha\colon \Gamma\to \Diff^1(M)$ as in  Theorem \ref{main2},  the restricted action $\restrict {\alpha}{\Lambda_{1,2}}\colon \Gamma\to \Diff^1(M)$ has  uniform subexponential growth of derivatives.
\end{proposition}

\subsection{Orbits with large fiber growth yet low depth in the cusp}\label{maximal}


To prove Proposition \ref{prop:maybeweshouldstatethemainresultatsomepoint}, as in Section \ref{sec:mutualmastication} we consider a canonical embedding $X=H_{1,2}/\Lambda_{1,2}$ of $\Sl(2,\R)/\Sl(2, \Z)$ in $\Sl(m,\R)/\Sl(m,\Z)$.
Write  $$a^t := \text{diag}(e^{t/2}, e^{-t/2})\subset \Sl(2,\R)$$ for the geodesic flow on $X$.
Let $X_{\text{thick}}$ be a fixed compact   $\So(2)$-invariant ``thick part'' of $X$;  {\blue that is, relative to the Dirichlet  domain $\mathcal D$ in \eqref{eq:dirc}, points in $\So(2)\backslash X_{\text{thick}}$  corresponds to the points in  $\So(2) \backslash \mathcal D$  whose imaginary part is bounded above, say, by   $17$.}


A geodesic curve in the modular surface of length $t$ corresponds  to the image of  an orbit $\zeta= \{a^s (x)\}_{0\leq s \leq t}$ where $x \in X$ and $t \geq 0$.  Denote the length of such a curve by $l(\zeta)$. For an orbit   $\zeta= \{a^s (x)\}_{0\leq s \leq t}$ of $\{a^t\}$ in $X$ we define $$c(\zeta) := \log (\|D_x(a^t)\|_{\text{Fiber}}).$$


The following claim is straightforward from the compactness  $X_{\text{thick}}$ and the quasi-isometry between the word and Riemannian metrics on $\Gamma$.  

\begin{claim}\label{slowgrowththick} For an action $\alpha\colon \Sl(m,\Z) \to \diff^1(M)$, the following statements are equivalent:
\begin{enumerate}
\item  the restriction $\restrict \alpha {\Lambda_{1,2}}\colon \Lambda_{1,2} \to \diff^1(M)$ has uniform subexponential growth of  derivatives;
\item for any $\e>0$ there is a $t_{\e}>0$ such that for any orbit  $\zeta =  \{a^s (x)\}_{0\leq s \leq t}$ with  $x\in \Xt$, $a^t(x) \in \Xt$, and $l(\zeta) = t \geq {t_\e}$ we have  $$ c(\zeta) \leq \e l(\zeta).$$
\end{enumerate}
\end{claim}

Define the maximal fiberwise growth rate of orbits starting and  returning to $\Xt$ to be \begin{equation}\label{eq:fanorlamps}\chi_{\mathrm{max}} := \limsup_{t > 0} \left\{\sup \left\{   \frac{\log \|\restrict{D_x(a^t)}{{\text{Fiber}}}\|}{t} : x\in \Xt, a^t(x) \in \Xt\right\}\right\}.\end{equation} 
Using Claim \ref{slowgrowththick}, to establish Proposition \ref{prop:maybeweshouldstatethemainresultatsomepoint} it is sufficient to show  that $\chi_{\mathrm{max}} = 0$.

For an orbit  $\zeta = \{a^s (x)\}_{0\leq s \leq t}$,   define the following function which measures the depth of $\zeta$ into the cusp:
$$d(\zeta) = \max_{0 \leq s \leq t} \text{dist}(a^s(x), \Xt).$$

The following lemma is the main result of this subsection.

\begin{lemma}\label{lemma:maximal} If  $\chi_{\text{max}} >0$ then  there exists a sequence of orbits $\zeta_n=\{a^s (x_n)\}_{0\leq s \leq t_n}$ with  $x_n\in \Xt$, $a^{t_n}(x_n) \in \Xt$, and   $t_n = l(\zeta_n) \to \infty$ 
such that
\begin{enumerate}
\item  $\displaystyle c(\zeta_n) \geq \frac{\chi_{\mathrm{max}}}{2} t_n$;
\item  $\displaystyle \lim_{n \to \infty} \frac{d(\zeta_n)}{t_n} = 0$.
\end{enumerate}
\end{lemma}



We first have the following claim.
\begin{claim}\label{pathcusp}
 For any $\e>0$ there exists $t_{\e}$ with the following properties: for any $x\in \partial \Xt$ and $t\ge t_\epsilon$ such that  $a^s(x)\in  X\sm \Xt$ for all $0< s < t$ and $a^t(x) \in \partial \Xt$ then, for the orbit $\zeta =  \{a^s (x)\}_{0\leq s \leq t}$, we have
$$c(\zeta) \leq \e t = \epsilon l(\zeta).$$
\end{claim}

Indeed, the claim follows from the fact that the value of the  return cocycle $\beta(a^s, x)$ is defined by geodesic in the cusp of $X$ is given by  a unipotent matrix of the form $\left(\begin{array}{cc}1 & n \\0 & 1\end{array}\right)\in \Lambda_{1,2}\subset \Sl(m,\Z)$ and  Proposition \ref{unipotentisgood}.

\begin{proof}[Proof of Lemma \ref{lemma:maximal}] Let $\zeta_n:= \{a^s (x_n)\}_{0\leq s \leq t_n}$ be a sequence of orbits with $x_n\in \Xt$, $a^{t_n}(x_n)\in \Xt$,  $t_n\to \infty$, and  such that  $$\chi_{\max} = \lim_{n \to \infty}\frac{ c(\zeta_n)}{ t_n}.$$
Replacing $\zeta_n$ with a subsequence, we may assume the following limit exists: $$\beta := \lim_{n \to \infty} \frac{d(\zeta_n)}{t_n}.$$

We aim to prove that $\beta = 0$. Arguing by contradiction,  suppose $0< \beta\le 1$. We decompose the orbit $$\zeta_n = 
\alpha_{k_n}\omega_{k_{n-1}}\alpha_{k_{n-1}}\cdots \omega_{1}\alpha_1$$ 
as a concatenation of smaller orbit segments $\alpha_i, \omega_i$ with the following properties:
\begin{enumerate}
\item each orbit $\alpha_i$ is  such that $d(\alpha_i) \leq \frac{\beta}{2} t_n$;
\item the endpoints of each orbit $\alpha_i$ are contained in $\Xt$;
\item  each orbit $\omega_i$ is  contained entirely in $(X\sm \Xt) \cup \partial \Xt$ with endpoints contained in $\partial \Xt$;
\item each orbit $\omega_i$  satisfies $d(\omega_i) \geq \frac{\beta}{2} t_n$  whence $l(\omega_i)\ge \frac {\beta}{2} t_n$ for $t_n$ sufficiently large.
\end{enumerate}
Note for each  $n$, that $k_n \leq \lfloor \frac{2}{\beta}\rfloor+ 1$ and  thus $k_n$ is   bounded by some $k$ independent of $n$. Additionally,  since   $\Sl(m, \Z)$  is finitely generated and (equipped with the word metric) is quasi-isometrically embedded  in  $\Sl(m, \R)$,  there exists a constant $K$ such that for any orbit  segment $\zeta$ whose endpoints are contained  in $\Xt$, we have  $c(\zeta) \leq K l(\zeta)$.
By the definition of $\chi_{\max}$, for any $\e> 0$ there is a positive constant $M_{\e}$ such that for any orbit sub-segment $\alpha_i$ 
\begin{enumerate}
\item $c(\alpha_i) \leq (\chi_{\max} + \e)l(\alpha_i)$    whenever  $l(\alpha_i) > M_\e$
\item $c(\alpha_i) \leq KM_\e$ whenever $l(\alpha_i) \leq M_{\e}$.
\end{enumerate}
From Claim  \ref{pathcusp}, for any   $\e>0$ we have, assuming that $n$ and hence $t_n$ are   sufficiently large,   that $$c(\omega_i) < \epsilon l(\omega_i)$$ for all orbit sub-segmants $\omega_i$.

Taking $n$  sufficiently large  we have
\begin{equation}\label{eqmax1}
(\chi_{\max} - \e)t_n < c(\zeta_n) \leq \sum_i c(\omega_i) + \sum_i c(\alpha_i).
\end{equation}
As we assume $\beta>0$, for all sufficiently large $n$ there exists at least one orbit sub-segment $\omega_i$ and thus for such $n$
\begin{equation}\label{eqmax2}
\sum_i c(\alpha_i) \leq kKM_{\e} + (\chi_{\max} + \e)\sum_i l(\alpha_i) \leq kKM_{\e} + (\chi_{\max} + \e)(1-  \beta/2   )t_n.
\end{equation}
From   \eqref{eqmax1} and \eqref{eqmax2} we obtain that
\begin{equation}
(\chi_{\max} - \e)t_n \leq   (\epsilon t_n) + \Big(kKM_{\e} + (\chi_{\max} + \e)(1- \beta /2)t_n\Big).
\end{equation}
Dividing by $t_n$ and taking $n \to \infty$ obtain
$$ \chi_{\max} - \e \leq \e  + (\chi_{\max} + \e)(1-  \beta/2).$$
As we assumed $\chi_{\max}>0$ and $\beta > 0$, we obtain a contradiction by taking $\e>0$ sufficiently small.  
\end{proof}

\subsection{Construction of a \Folner sequence and family averaged measures}
\label{section:folner}
Assuming that $\chi_{\mathrm{max}}$ in \eqref{eq:fanorlamps} is non-zero, we start from the orbit segments constructed in Lemma \ref{lemma:maximal} and perform an averaging procedure to obtain a family of measures $\{  \mu _n\}$ on  $M^\alpha$ whose properties lead to a contradiction.   In particular, 
the projection of any weak-$*$ limit $  \mu _\infty$ of $  \mu _n$  to $M^\alpha$ will be $A$-invariant, well behaved at the cusps, and have non-zero Lyapunov exponents.  These measures on $M^\alpha$ are obtained by averaging certain Dirac measures against \Folner sequences in a certain amenable subgroup of $G$.

Consider the copy of  $\Sl(m-1, \R)\subset \Sl(m,\R)$ as the subgroup of matrices that differ from the identity away from the $m$th row and $m$th column.  Let $N'\simeq \R^{m-1}$ be the abelian  subgroup of unipotent elements that differ from the identity only in the $m$th column; that is given a vector $r = (r_1,r_2, \dots, r_{m-1})\in \R^{m-1}$ define $u^{r}$ to be the unipotent element
\begin{equation}\label{eq:formerlyknownas}
u^r= \left(\begin{array}{ccccc}1  &  0  & 0  & \dots  & r_1  \\  & 1  &   0 & \dots  & r_2  \\  &   &  \ddots &   &  \vdots  \\  &   &   &  1 &  r_{m-1} \\  &   &   &   &1  \end{array}\right)\end{equation}
and let $N'= \{u^r\}$.  $N'$ is normalized by $\Sl(m-1,\R)$.

Identifying $N'$ with $\R^{m-1}$ we have an embedding
   $\Sl(m-1, \R) \ltimes \R^{m-1}  \subset \Sl(m,\R)$.  
The subgroup $\Sl(m-1, \R) \ltimes \R^{m-1}$ has as a  lattice the subgroup $$\Sl(m-1, \Z) \ltimes \Z^{m-1}:= \Gamma \cap\big( \Sl(m-1, \R) \ltimes \R^{m-1}\big)$$ and there is a natural embedding given by the inclusion $$(\Sl(m-1,\R)\ltimes \R^{m-1}) / (\Sl(m-1, \Z^{-1}) \ltimes \Z^{m-1})) \subset \Sl(m,\R) / \Sl(m,\Z).$$

Recall $A$ is the group of diagonal matrices with positive entries.
Let $a^t, b^s \in A$ denote matrices
$$a^t = \text{diag}(e^{t/2}, e^{-t/2}, 1,1...,1)$$
$$b^s= \text{diag}(e^s, e^{s}, e^{s}..,e^{s}, e^{-s(m-1)}).$$
Complete the set  $\{a,b\}$  to a spanning set  $\{a,b, c_1, c_2 \dots c_{m-3}\}$ of $A$ viewed as vector space where the $c_i$ are diagonal matrices whose $(m,m)$-entry is equal to $1$.


Let $F_n\subset AN'$ be the subset of $G$ consisting of all the elements of the form
\begin{equation}\label{folner}
a^tb^{s} \prod_{c=1}^{m-3} c_i^{s_i} u^{r}
\end{equation}
where, for some $\delta>0$   to be determined later (in the proof of Proposition \ref{positiveexponent} below),
\begin{enumerate}
\item $0<t<t_n$;
\item$ \delta t_n/2<s < \delta t_n$;
\item $ 0 < s_i <  \sqrt{t_n}$;
\item $r \in B_{\R^{m-1}} (e^{200t_n})$.
\end{enumerate}

\begin{claim}
$\{F_n\}$ is \Folner sequence in $AN'$.
\end{claim}

Observe that $F_n$ is  linearly-long in the $a$-direction and exponentially-long in the $N'$-direction. From conditions (2) and (4), the $A$-component of $F_n$ is much longer in the $a^t$-direction than in the other directions.
The condition (2) that  $\delta t_n/2<s$ is fundamental in our estimates in Section \ref{section:goodcusps} that ensure the  measures constructed below $\{\mu_n\}$ have uniformly  exponentially small mass in the cusps. These estimates are related to the fact that orbits of $N'$ correspond to the   unstable manifolds for the flow defined by $b_{s}$ in $\Sl(m,\R)/\Sl(m,\Z)$ and open subsets of unstable manifolds equidistribute to the Haar measure on $\Sl(m,\R)/\Sl(m,\Z)$  under the flow $b^{s}$.

Recall we have a sequence of fiber bundles $$F\to M^\alpha\to G/\Gamma$$ 
and may consider $F$  
as a fiber bundle over $G/\Gamma$.  Given $x\in G/\Gamma$, let $F(x)\simeq TM$ denote the fiber of $F$ over $x$.  An element $v\in F(x)$ is a pair $v= (y,\xi)$ where, identifying the fiber of $M^\alpha$ through $x$ with $M$, we have  $y\in M$ and $\xi\in T_y M$.  Given $v= (y,\xi)\in F(x)$,  we write $\|v\| = \|\xi\|$ using our chosen norm on $F$.  
Given $v = (y,\xi)\in F(x)$, let $p(v) = y$ denote the footpoint of $v$ in the fiber of $M^\alpha$ through $x$.  

If  uniform subexponential growth of derivatives fails for the restriction of the  $\alpha$ to $\Lambda_{1,2}$, then  there exist sequences $x_n\in \Xt$,   $v_n\in F(x_n)$  with $\|v_n\| = 1$, and $t_n \in \R$
as in Lemma \ref{lemma:maximal} and Claim \ref{slowgrowththick} with $t_n\to \infty$,   such that \begin{equation}\label{eq:cleansingDeluge}\|D_{x_n}a_n^{t_n}(v_n)\| \geq e^{\lambda t_n}\end{equation} for some $\lambda >0$.

Note that $AN'$ is a solvable group.  We may equip $AN'$ with any left-invariant Haar measure.  Note that the ambient Riemannian metric induces a right-invariant Haar measure on $AN'$ but as $AN'$ is not unimodular these measures do not coincide.

For each $n$, take  $ \mu _n$ to be the measure on $M^\alpha$    obtained by averaging the Dirac measure $\delta_{(x_n,p(v_n))}$  over the set $F_n$: 
$$  \mu _n:= \frac{1}{|F_n|_\ell} \int _{F_n}  g \cdot \delta(x_n, p(v_n))  \ d g$$
where $|F_n|_\ell$ is the volume of $F_n$ and $dg $ indicates integration with respect to left-invariant Haar measure on  $AN'$.

We expand the above integral in our coordinates introduced above.
Then for   any bounded continuous function $f\colon M^\alpha  \to \R$, integrating against our Euclidean parameters $t,s, s_i,$ and $r$ we have
\begin{equation}\label{eq:lazyAaron}
\begin{aligned}
\int\limits_{M^\alpha } &f \ d  \mu_n
\quad \quad
\\
&=  \frac{\displaystyle 2 \int\limits_{0}^{t_n} \int\limits_{\delta t_n /2}^{\delta t_n} \int\limits_{[0, \sqrt{t_n}]^{m-3}}
\int\limits_{B_{\R^{m-1}}(e^{200t_n})} f\left (a^tb^s\prod_{c=1}^{m-3} c_i^{s_i}u^r \cdot (x_n, p(v_n) )\right ) \  dr   \ d{s_i}\ ds\ dt  }{{t_n}{\delta t_n}{\sqrt{t_n}}^{m-3}{|B_{\R^{m-1}}(e^{200t_n})|}}
 \end{aligned}
 \end{equation}


 \noindent where $|B_{\R^{m-1}}(e^{200t_n})|$ denotes the volume of $$B_{\R^{m-1}}(e^{200t_n})= N'_{t_n} \subset  N'$$ with respect to the Euclidean parameters $r$.

%
%





For each $n$, let $\nu_n$ denote the image of the measure $\mu_n$ under the canonical projection from  $M^\alpha$ to $G/\Gamma$.
The following proposition is  shown in the next subsection.
\begin{proposition}\label{maincusps}
There exists $\eta>0$ such that the sequence of measures $\{\nu_n\}$ has uniformly exponentially small mass in the cusp with exponent $\eta$.
\end{proposition}
By the uniform comparability of distances in fibers of $M^\alpha$, this implies the family of measures $\{ \mu_n\}$ has uniformly exponentially  small measure in the  cusp.

By Lemma \ref{lemma:firstexponents}\ref{lazylemmaa} the families of measures    $\{\mu_n\}$  and $\{\nu_n\}$ are precompact families.
 As $F_n$ is a \Folner sequence in a solvable group, we have that
any weak-$*$ subsequential limit  of  $\{\mu _n\}$  or $\{\nu_n\}$   is  $AN'$-invariant.
Moreover, from Theorem \ref{thm:ratner}\ref{ratner3}, it follows that any  weak-$*$ subsequential limit $\nu_\infty$ of $\{\nu_n\}$ is invariant under the group $-N'$ generated by the root groups $U^{m,j}$ for each $1\le j\le m-1$.  Since $N'$ and $-N'$ generate all of $G$, we have that $\nu_\infty$ is a $G$-invariant measure on $G/\Gamma$.

\subsection{Proof of Proposition \ref{maincusps}}
\label{section:goodcusps}
\subsubsection{Heuristics of the proof} The heuristic of the proof is the following.  Observe that for a fixed choice of $t$ and $s_i$ as given by the choice of \Folner set $F_n$, the point $$a^t\prod_{i=1}^{m-3}{c_i}^{s_i}(x_n)$$ lies at sub-linear distance to the thick part of $G/\Gamma$ with respect to $t_n$. Observe that the $N'$-orbit of such point is an embedded $(m-1)$-dimensional torus in $G/\Gamma$. As the range of points in $N'$ in the  \Folner set $F_n$ is quite large,  averaging a Dirac measure of the point $a^t\prod_{i=1}^{m-3}{c_i}^{s_i}(x_n)$ in the $N'$-direction in $F_n$ yields a measure quite close to Haar measure on the $N'$-orbit.

Observe that $N'$-orbits correspond to unstable manifolds for the action of the flow $b^s$ on  $\Sl(m,\R)/\Sl(m,\Z)$. As the action of $b_s$ is known to be mixing, we expect that if  $s$ is sufficiently large, flowing by $b_s$ the $N'$-orbit of $a^t\prod_{i=1}^{m-3}{c_i}^{s_i}(x_n)$ will become equidistributed and in particular it will intersect non-trivially the thick part of $G/\Gamma$. This is the reason why the condition $s > \delta/2 t_n$ is assumed.

While intuition about mixing motivates the proof, we do not use it explicitly.  Instead we use that for large enough $s$, the action of $b_s$ expands the $N'$-orbits in a way that forces them to hit the thick part. We verify this fact by explicit matrix multiplication.

 As $b^s$ normalizes $N'$, the image under $b^s$ of the $N'$-orbit of $a^t\prod_{i=1}^{m-3}{c_i}^{s_i}(x_n)$ is the $N'$-orbit of a point $y_n$ in the thick part of $G/\Gamma$.
 Having in mind the quantitative non-divergence of unipotent flows as in the proof Proposition \ref{prop:bananas},   the $N'$-orbits  have uniformly (over all $n, s_i, $ and $t$) exponentially small mass in the cusps whence so do the measures $\nu_n$. 


 The following  proof of  Proposition \ref{maincusps} uses   explicit matrix calculations and estimates to verify  these heuristics.


\subsubsection{Proof of Proposition \ref{maincusps}}
Recall that we identify   each coset $$g\Sl(m,\Z) \in \Sl(m,\R)/\Sl(m,\Z)$$ with a unimodular lattice
$\Lambda_g:= g\cdot \Z^m$ in $\R^m$.  We define the systole of a unimodular lattice $\Lambda\subset \R^m$ to be  $$\delta(\Lambda) := \min_{v \in \Lambda \setminus \{0\}} \|v\|$$ and for an element $g\in \Sl(m,\R)$,   we denote by $\delta(g)$ the systole $$\delta(g) = \delta(g\cdot \Z^m).$$ From \eqref{eq:ploy}, to prove Proposition \ref{maincusps} it is sufficient to find $\eta>0$ so that the  integrals $$\int_{G/\Gamma} \delta(g)^{-\eta}  \ d \nu_n(g\Gamma)$$
are uniformly bounded in $n$.

As discussed in the above heuristic, from \eqref{eq:lazyAaron} to bound the integrals $\int_{G/\Gamma} \delta^{-\eta}(g)  \ d \nu_n(g\Gamma)$ it is sufficient to show each integral $${\frac{1}{|B(e^{200t_n})|}}\int_{B(e^{200t_n})}  \delta(a^tb^s (\Pi {c_i}^{s_i})u^r   x_n)^{-\eta} \  dr$$
is uniformly bounded in $n$ and in all parameters $t,s, s_i$ for $0<t<t_n,$
$ \delta t_n/2<s < \delta t_n,$ and $
0 < s_i <  \sqrt{t_n}.$
Recall here that $  x_n\in G/\Gamma$ are the points  $x_n\in \Xt\subset H_{1,2}/\Lambda_{1,2}$   satisfying \eqref{eq:cleansingDeluge} used in the construction of the measures $\mu_n$.

{\blue We have $H_{1,2}$ is canonically embedded in $\Sl(m,\R)$.   Given $x_n\in H_{1,2}/\Lambda_{1,2}$, let $$\td x_n\in H_{1,2}\subset \Sl(m,\R)$$ denote the element   mapping to $  x_n$ under the map $H_{1,2}\to H_{1,2}/\Lambda_{1,2}$ which is contained in a fundamental domain contained in the Dirichlet domain $\mathcal D\subset \SL(2,\R)$  in \eqref{eq:dirc} in Section \ref{sec:lolo}.}  Let $\|\cdot \|$ denote the operator norm on $\Sl(m,\R)$ and $m(\cdot )$ the associated conorm.



%

%



\begin{claim} For every $n$,  $t\le t_n$, and $ 0\le s_i\le \sqrt{t_n}$ as above, there exist $$\text{$A_{n}= A_{n,t,s_1, \dots , s_{m-3}} \in \Sl(m-1, \R)$ and $\gamma_{n} = \gamma_{n,t,s_1, \dots , s_{m-3}} \in \Sl(m-1,\Z)$}$$ such that:

\begin{enumerate}\label{An}

\item $ \displaystyle a^t \prod_{i=1}^{m-3}{c_i}^{s_i} \td x_n =
\begin{pmatrix}
  A_n\gamma_n & 0_{m-1 \times 1} \\ 0_{1\times m-1} & 1
\end{pmatrix}
$

\item $ \displaystyle\lim_{n\to \infty} \sup_{t\le t_n,   0\le s_i\le \sqrt{t_n}}  \frac{\log\|A_n\|}{t_n} = 0 $ and  $ \displaystyle\lim_{n\to \infty} \inf_{t\le t_n,   0\le s_i\le \sqrt{t_n}}  \frac{\log(m(A_n))}{t_n} = 0 $

\end{enumerate}

\end{claim}


\begin{proof}(1) is immediate from construction.  The uniform limit in (2) follows from Lemma \ref{lemma:maximal}(2), equation \eqref{eq:easy}, and the fact that the  $s_i$ are chosen so that $0\leq s_i \leq \sqrt{t_n}$ whence $$\frac{d(  x_n , a^t  (\Pi {c_i}^{s_i}) \cdot   x_n )}{t_n}\to 0$$ uniformly in $t, s_i$.
\end{proof}

In the remainder, we will suppress the dependence of  choices on $t, s, s_i$.
We take  $K_n \in \Sl(m-1, \R)$ be such that
\[ \td x_n =
\begin{pmatrix}
  K_n & 0_{m-1 \times 1} \\ 0_{1\times m-1} & 1
\end{pmatrix}.
\]
{\blue Note that $K_n$ differs from the identity only in the first two rows and columns.}
Since each $x_n$ is contained in  $  \Xt$,  we have that the matrix norm  and conorm  $\|K_n\|$ and $m(K_n)$ are bounded above and below, respectively,  by   constants $M_1$ and $\frac 1 {M_1}$ independent of $n$.

Recall $r$ denotes a vector in $\R^{m-1}$ and $u^r\in \Sl(m,\R)$ is the unipotent element given by \eqref{eq:formerlyknownas}.
Matrix computation yields
\[ a^t  (\Pi {c_i}^{s_i})u^r \td  x_n =
\begin{pmatrix}
  A_n\gamma_n & A_n\gamma_{n}K_{n}^{-1}r \\ 0_{1\times m-1} & 1
\end{pmatrix}
\]
whence
\[ b^{s}(\Pi {c_i}^{s_i})a^tu^r \td x_n =
\begin{pmatrix}
  e^{s}A_n\gamma_n & e^{s}A_n\gamma_nK_{n}^{-1}r \\ 0_{1\times m-1} & e^{-(m-1)s}
\end{pmatrix}.
\]


We have \begin{equation}\label{integer}
\begin{split}
\delta( b^{s}(\Pi {c_i}^{s_i})a^tu^r \td x_n) &= \delta( b^{s}(\Pi {c_i}^{s_i})a^tu^r    x_n) \\&= \inf_{z \in \Z^m \setminus \{0\}} \bigg{\|}
\begin{pmatrix}
  e^{s}A_n\gamma_n & e^{s}A_n\gamma_nK_n^{-1}r \\ 0_{1\times m-1} & e^{-(m-1)s}
\end{pmatrix}
z
\bigg{\|}
\end{split}
\end{equation}
 To reduce notation, for fixed $t,s,$ and $s_i$   define $$\beta(r) :=- \log \delta (a^tb^s(\Pi {c_i}^{s_i})u^r \td x_n).$$ We aim to find an upper bound of $$\frac{1}{|B(e^{200t_n})|} \int_{B(e^{200t_n})} e^{\eta \beta(r)} dr$$ that is independent of $n$ and $ t,s,$ and $s_i$.


Observe that if $r - {r'} $ differ by an element of the unimodular lattice $ K_n \Z^{m-1}\subset \R^{m-1}$, then $\beta(r) = \beta(r')$.
{\blue
Indeed, if $r'= r+ K_n z'$ for some $z'= (z_1', \dots, z_{m-1}')\in \Z^{m-1}$ and if $z\in \Z^m\sm\{0\}$ is $z= (z_1, \dots , z_m)$ then
$$\begin{pmatrix}
  e^{s}A_n\gamma_n & e^{s}A_n\gamma_nK_n^{-1}r' \\ 0_{1\times m-1} & e^{-(m-1)s}
\end{pmatrix}z=\begin{pmatrix}
  e^{s}A_n\gamma_n & e^{s}A_n\gamma_nK_n^{-1}r \\ 0_{1\times m-1} & e^{-(m-1)s}
\end{pmatrix} \td z
$$
where $\td z = (z_1 + z_m z'_1, \dots, z_{m-1} + z_m z'_{m-1}, z_m)\in \Z^m\sm\{0\}.$
} 
Thus   we have that $\beta\colon \R^{m-1}\to (0,\infty)$ descends to a function on the torus $\R^{m-1}/(K_n \Z^{m-1})$.

Let $D_n = K_n \cdot ( [-1/2,1/2]^{m-1})$ be a fundamental domain for this torus in $\R^{m-1}$ centered at $0$.
Let $c_n$ denote the number of $(K_n\Z^{m-1})$-translates of $D_n$ that intersect $B(e^{200t_n})$.
Then, if $t_n$ is sufficiently large we have that
\begin{equation*}
 \frac{1}{|B(e^{200 t_n})|}\int_{B(e^{200t_n})} e^{\eta \beta(r)} \ dr
  \le
  \frac{1}{|B(e^{200 t_n})|}c_n \int_{D_n} e^{\eta \beta(r)} \ dr
  \le 2\int_{D_n} e^{\eta \beta(r)} \ dr
 \end{equation*}
 The first inequality follows from inclusion.  The second inequality follows from the fact that the perimeter of $B(q)$ grows like $q^{m-2}$, the volume of $B(q)$ grows like $q^{m-1}$, and the domains $D_n = K_n  \cdot ( [-1/2,1/2]^{m-1})$ have uniformly comparable geometry over $n$.

%
%
%

It remains to estimate $\int_{D_n} e^{\eta \beta(r)} \ dr$.  Given $c>0$ and fixed  $n, t, s_i, $ and $s$  we define $$T_c  = \{  r \in D_n:\beta(r) > c\}.$$
Proposition \ref{maincusps}   follows immediately  from the estimate in the following lemma.
\begin{lemma}\label{Tc} 
There exists constants $M_3,M_4 >0$, independent of $n, t, s_i, $ and $s$,    such that
$$ |T_c| \leq M_3e^{-cM_4}.$$
\end{lemma}
Indeed, if $\eta\inv > { M_4}$ then
\begin{align*}\int_{D_n} e^{\eta \beta(r)} \ dr
&=\int_0^\infty |\{ r\in D_n: e^{\eta \beta(r) }\ge \tau\}|  \ d \tau
\le 1+ \int_1^\infty |\{ r\in D_n: e^{\eta \beta(r) }\ge \tau\}|  \ d \tau\\
&=1+ \int_1^\infty |\{ r\in D_n:   \beta(r) \ge \log\big(\tau^{\frac 1 \eta}\big)\}| \ d \tau
=1+\int_1^\infty | T_{\log\big(\tau^{\frac 1 \eta}\big)}| \ d \tau\\
&= 1 + \int_1^\infty M_3  \tau^{\frac { -M_4} \eta} \ d \tau <\infty
\end{align*}
 and Proposition \ref{maincusps} follows.

\begin{proof}[Proof of Lemma \ref{Tc}]
From  \eqref{integer}, given any   $r \in \R^{m-1}$, if $\beta(r) > c$ then there exists a  non-zero $z = (z_1,z_2,z_3, .... , z_{m}) \in \Z^m$ such that
\[ e^s\left\| A_n\gamma_n
(z_1, \dots, z_{m-1})
 +
 z_mA_n\gamma_n K_n\inv
r
\right\|  < e^{-c}
\ \ \text{ and } \ \
|z_m| < e^{-c} e^{(m-1)s}
\]
which (as $\gamma_n \in \Sl(m-1, \Z)$) holds if and only if there is a non-zero $z = (z_1,z_2,z_3, .... , z_{m}) \in \Z^m$
\begin{equation}\label{eq:pleasetakeafreshmanlogiccourse} e^s\left\| A_n \Big(
(z_1, \dots, z_{m-1})
 +
 z_mK_n\inv (K_n\gamma_n K_n\inv) r
\Big)
\right\| < e^{-c}
\ \ \text{ and } \ \
|z_m| < e^{-c} e^{(m-1)s}
\end{equation}
As $K_n\gamma_n K_n\inv$ induces a volume-preserving automorphism of $\R^{m-1}/(K_n\Z^{m-1})$, the set of $r\in D_n$ satisfying \eqref{eq:pleasetakeafreshmanlogiccourse}  {\blue for some $z\in \Z^m$} has the same measure as the set of $r\in D_n$ satisfying
$$e^s\left\| A_n \Big(
(z_1, \dots, z_{m-1})
 +
 z_mK_n\inv   r
\Big)
\right\| < e^{-c}
\ \ \text{ and } \ \
|z_m| < e^{-c} e^{(m-1)s}$$
{\blue for some $z\in \Z^m.$}

For every integer $k$ satisfying  $|k| < e^{-c} e^{(m-1)s}$, let $T_{c,k}$ be the subset of $r\in D_n$ such that there exists $(z_1,z_2, \dots, z_{m-1})\in \Z^{m-1}$  satisfying
$$e^s\left\| A_n \Big(
(z_1, \dots, z_{m-1})
 +
k K_n\inv   r
\Big)
\right\|< e^{-c}.$$
 Then $|T_c| \leq   \sum_{|k| < e^{-c} e^{(m-1)s}} |T_{c,k}|.$   Thus the estimate reduces to the following.




\begin{claim}\label{rennes} There exists $M_5 \geq 0$ such that $|T_{c,k} |< M_5 e^{-(m-1)(s+c)}$  for all $n$ sufficiently  large.
\end{claim}
\begin{proof}Recall that $\delta t_n/2 <s$.
If $k = 0$ then,   for any non-zero $(z_1, \dots,  z_{m-1})\in \Z^{m-1}$, we have \[ e^s\left \| A_n
(z_1, \dots,  z_{m-1})
\right \| > e^{\delta t_n/2} m(A_n).
\]
From Claim \ref{An}(2),  if $n$ is large enough then  $m(A_n)\ge e^{-\delta t_n/4}$ and so the term in the left hand side above is greater than one, therefore $T_{c,0} = \emptyset$ for $n$ sufficiently large.

If $k \neq 0$,
observe that the map $M_k\colon \R^{m-1}/K_n \Z^{m-1} \to \R^{m-1}/K_n\Z^{m-1}$ given by $$r+ K_n \Z^{m-1}\mapsto kr+ K_n \Z^{m-1}$$ preserves the Lebesgue measure on $\R^{m-1}/K_n \Z^{m-1}.$  In particular, this implies that $T_{c,k}$ and $T_{c,1}$ have the same volume.

{\blue
We thus take $k = 1$.
Note that $K_n\inv D_n=[-1/2,1/2]$.   There is a $L\ge 1$, depending only on $m-1$, such that   the set
	$$Q=\{z'\in \Z^{m-1} :  |z'+r|\le 1 \text{ for some $r\in K_n\inv D_n$ }\}$$
	has cardinality at most $L$.    From Claim \ref{An}(2),  if $n$ is large enough then  $m(A_n)\ge e^{-\delta t_n/4}$ whence for all $ (z_1, \dots, z_{m-1}) \in \Z^{m-1} \sm Q$ and all $r\in D_n$,
	$$e^s\left\| A_n \Big(
 (z_1, \dots, z_{m-1}) + K_n\inv   r\Big)
\right\|  \ge 1.$$
We thus need only consider $(z_1, \dots, z_{m-1})\in Q$.
}


Given a fixed $z= (z_1, \dots,  z_{m-1})\in Q\subset  \Z^{m-1}$, using that $K_n\in \Sl(m-1, \R)$
we have $$\left | \{r\in \R^{m-1}  : \|z+K_n\inv  r\|\le \ell \} \right|\le (2\ell )^{m-1}$$
whence $$\left |\{ r\in \R^{m-1}  :  \|(z_1, \dots, z_{m-1}) + K_n\inv   r \| \le e^{-c}  \} \right| \le 2^{m-1} e^{-c(m-1)}.$$

If $r\in T_{c,1}$ so that
$$e^s\left\| A_n \Big(
 (z_1, \dots, z_{m-1}) + K_n\inv   r\Big)
\right\|  \le e^{-c}$$
then
\begin{equation}\label{eq:thisiswhatissoundslikewhenpostdocscry}\left\| A_n \Big(
 (z_1, \dots, z_{m-1}) + K_n\inv   r\Big)
\right\|  \le e^{-c-s}.\end{equation}
Since $A_n\in \Sl(m-1,\R)$ the set of $r\in \R^{m-1}$ satisfying  \eqref{eq:thisiswhatissoundslikewhenpostdocscry}
has the same volume as the set of $r\in \R^{m-1}$ satisfying
\[\left\|
 (z_1, \dots, z_{m-1}) + K_n\inv   r \right\|  \le e^{-c-s}. \]
{\blue  It follows that $|T_{c,1}|\le   2^{m-1} L e^{-(s+c)(m-1)}$.}\end{proof}

To finish the proof of Lemma \ref{Tc}, from Claim \ref{rennes} we have $$|T_c| \leq \sum_{|k| < e^{-c} e^{(m-1)s}} |T_{c,k}| \leq (2e^{-c} e^{(m-1)s}+1) M_5 e^{-(m-1)(s+c)} \leq M_3e^{-cM_4}$$ for some constants $M_3, M_4$ independent of $n$.
\end{proof}

\subsection{Positive Lyapunov exponents for limit measures}
\label{section:ANLyapunov}
To deduce Proposition \ref{prop:maybeweshouldstatethemainresultatsomepoint}, having assumed that $\chi_{\max}$ in \eqref{eq:fanorlamps} is non-zero, we show that  any {weak-${*}$ subsequential limit of the sequence of measures $\{\mu_n\}$ has a positive Lyapunov exponent from which we derive a contradiction.

Recall from Section \ref{section:folner} that we  fixed sequences $x_n, v_n, t_n$ such  that  $\|D_{x_n}a^{t_n}(v_n)\| \geq e^{ \lambda t_n}$ for some fixed $\lambda > 0$.
Let $\calA\colon G\times F\to F$ be the fiberwise derivative cocycle over the action of $G$ on $M^\alpha$.

Our main result is the following.
\begin{proposition}\label{positiveexponent}
For any weak-$*$ subsequential limit $\mu_\infty$ of $\{\mu_n\}$ we have
$$\lambda_{\top, a, \mu_\infty,\calA}\ge \lambda/2 >0.$$
\end{proposition}

%
%
%
%




We first show that averaging over $N'$ does not change the Lyapunov exponents of the cocycle.

\begin{claim}\label{siegelunipotent} Given any $\e> 0$ there is $t_{\e}>0$ such that for any $t \geq t_{\e}$ and any $r \in B_{\R^{m-1}}(e^{t})$ we have  $$\| {D_{x}u^r}\|_{\Fib} \leq e^{\e t}$$ 
for any $x\in \Xt$.
\end{claim}
\begin{proof}
 Recall that the $N'$-orbit of any $x\in X:= H_{1,2}/\Lambda_{1,2}\subset \Sl(m,\R)/\Sl(m,\Z)$ is a closed torus.    Then the $N'$-orbit of $\Xt$ is compact.
{\blue Recall our fixed fundamental domain $\calF\subset \wtd {\mathcal D}$ contained in the Dirichlet  domain  $\wtd {\mathcal D}$ of the identity for $\Sl(m,\R)/\Sl(m,\Z)$ as discussed in Section \ref{sec:lolo}.} Given $x\in  \Sl(m,\R)/\Sl(m,\Z)$, let $\td x$ be the lift of $x$ in $\calF$.  Let  $\wtd X _{\mathrm{thick}}\subset H_{1,2}\cap \calF$ denote the lift of $\Xt$ to  $\calF$ and let $\hat X _{\mathrm{thick}}$ be the lift of the orbit $N'\Xt$ to $\calF$.  {\blue As discussed in Section \ref{sec:lolo}, we have that $\wtd X _{\mathrm{thick}}$ is contained in the Dirichlet domain $\mathcal D$ of the identity for the  $\Lambda_{1,2}$-action on $H_{1,2}$.  Moreover, $\hat X _{\mathrm{thick}}$ is precompact in $\Sl(m,\R)$.}

Fix $r\in \R^{m-1}$ and $x\in \Xt$.   Write $$\td x = \left(\begin{array}{cc}K& 0 \\0 & 1\end{array}\right)$$ for some $K\in \Sl(m-1, \R)$; we have $\|K\|\le M_1$ and $m(K) \ge \frac 1 {M_1}$ for all $x\in \Xt$.     The deck group of the orbit $N'\td x$ is $$\td x   \{ u^z : z\in \Z^{m-1}\} \td x\inv = \{  u^{K\cdot z} : z\in \Z^{m-1}\} .$$  Thus, there is $z\in \Z^{m-1}$ and $r'\in \R^{m-1}$ such that
$$u^r \td x = \left(\begin{array}{cc}K& r \\0 & 1\end{array}\right)= \left(\begin{array}{cc}K& r' +K z\\0 & 1\end{array}\right)= \left(\begin{array}{cc}1& r' \\0 & 1\end{array}\right) \left(\begin{array}{cc}K& 0\\0 & 1\end{array}\right) \left(\begin{array}{cc}1& z \\0 & 1\end{array}\right) = u^{r'} \td x u^z$$
and $u^{r'} \td x\in \hat X _{\mathrm{thick}}$.
Then
$$\|D_{x}u^r\|\le \|D_{x}\td x\inv \| _{\Fib} \cdot  \| D_{\Id \Gamma} u^z\| _{\Fib}  \cdot \| D_{\Id\Gamma} u^{r'}\td x\| _{\Fib}.$$
Since $\td x$ and $ u^{r'}\td x$ are in precompact sets, the first and last terms of the right hand side are uniformly bounded in $r$ and $x\in \Xt$.

There exists some $C$ such that
$$\| D_{\Id \Gamma} u^z\| _{\Fib}\le C \|D\alpha(u^z)\| .$$
Since $r\in B_{\R^{m-1}}(e^t)$ we  have $z\in B_{\R^{m-1}} (M_1 e ^t)$ whence $d(u^z,\id)\le C_2 t + C_3$ for some constants $C_2$ and $C_3$.
  Proposition \ref{unipotentisgood} implies for any $\e'$ that   $$\|D\alpha(u^z)\|  \leq e^{\e'  (C_2 t + C_3)}$$
and taking $\e'>0$ sufficiently small, the claim follows.
\end{proof}

By Lemma \ref{lemma:fromlmr}, the fact that $\Sl(m,\Z)$ is finitely generated, and the uniform comparability of the fibers of $M^\alpha$,  we also have the following.
\begin{claim}\label{claim90} There are uniform constants $C_5$ and $C_6$ with the following property:
Let $x\in G/\Gamma$.  Then for any $X\in \lieg$ with $\|X\|\le 1$ we have $$\left\|\big(D_{x} \exp (tX) \big)^{\blue \pm 1} \right\|_{\Fib}\le e^{C_5 t + C_5 d(x, \id) + C_6}.$$
\end{claim}

We now prove Proposition \ref{positiveexponent}.
\begin{proof}[Proof of Proposition \ref{positiveexponent}]
Recall we take $x_n\in \Xt$, $t_n\to \infty $, and $v_n \in F(x_n)$ with $\|v_n\| =1$ such that
$\|D_{x_n}a^{t_n} (v_n)\| \geq e^{\lambda t_n}$
for some fixed $\lambda>0$  in \eqref{eq:cleansingDeluge} in Section \ref{section:folner}.    We  also write $\calA\colon G\times F\to F$ for the fiberwise derivative cocycle.

The measures $\mu_n$ constructed in Section \ref{section:folner} are   defined by averaging last along the orbit $a^t, 0\le t\le t_n$.
Let $\xi_n$ be the measure on $M^\alpha$ given by
\begin{equation*}
\begin{aligned}
\int_{M^\alpha } &f \ d \xi_n ={  \frac{2}{\delta t_n}\bigg(\frac{1}{\sqrt{t_n}} }\bigg)^{m-3}\frac{ 1}{|B_{\R^{m-1}}(e^{200t_n})|}
\\&\quad \quad     \int_{\delta t_n /2}^{\delta t_n} \int_{[0, \sqrt{t_n}]^{m-3}}
\int_{B_{\R^{m-1}}(e^{200t_n})} f\left (a^tb^s\prod_{c=1}^{m-3} c_i^{s_i}u^r \cdot (x_n, p(v_n) )\right ) \  dr   \ d{s_i}\ ds.
 \end{aligned}
 \end{equation*}

In the context of  Lemma \ref{lemma:firstexponents}, the measures $\mu_n= \int_0^{t_n} (a^t_*\xi_n ) \ d t$ constructed in Section \ref{section:folner} correspond to the empirical measures $\eta_n= \eta(\log a, t_n, \xi_n)$ appearing in the proof of Lemma \ref{lemma:firstexponents}.  From Lemma \ref{lemma:firstexponents}, to establish
 Proposition \ref{positiveexponent} it is sufficient to show that
$$\int \log  \|\calA(a ^{t_n},\cdot) \|  \ d \xi_n 
\geq \frac{\lambda}{2} t_n.$$

We have
\begin{equation*}
\begin{aligned}
 \int_{M^\alpha }  & {\blue  \log  \|\calA(a ^{t_n},\cdot) \|  \ d \xi_n} \\
&\quad  ={  \frac{2}{\delta t_n}\bigg(\frac{1}{\sqrt{t_n}} }\bigg)^{m-3}\frac{ 1}{|B_{\R^{m-1}}(e^{200t_n})|}
\\ &\quad   \quad
    \int_{\delta t_n /2}^{\delta t_n} \int_{[0, \sqrt{t_n}]^{m-3}}
\int_{B_{\R^{m-1}}(e^{200t_n})}
{\blue \log \left\| \calA\big(  a^{t_n}, b^s\Pi  c_i^{s_i}u^r \cdot (x_n, p(v_n) )\big) \right\|\  dr \ ds_i \ d s}\\
&\quad  \ge
{  \frac{2}{\delta t_n}\bigg(\frac{1}{\sqrt{t_n}} }\bigg)^{m-3}\frac{ 1}{|B_{\R^{m-1}}(e^{200t_n})|}
\\ &\quad \quad
    \int_{\delta t_n /2}^{\delta t_n} \int_{[0, \sqrt{t_n}]^{m-3}}
\int_{B_{\R^{m-1}}(e^{200t_n})}
\log \frac{\left \| D_{ x_n  } \big( a^{t_n} b^s\Pi  c_i^{s_i}u^r\big) ( v_n)\right  \|    }{\left \| D_{ x_n  } \big(   b^s\Pi  c_i^{s_i}u^r\big) ( v_n) \right  \|  }
\  dr \ ds_i \ d s.\\
 \end{aligned}
 \end{equation*}
Consider fixed $r,$ $s, $ and $ s_i$.  Take $r'\in \R^{m-1}$ such that $a^{t_n} u^r = u^{r'}a^{t_n} $.  Then
 \begin{equation*}  \begin{aligned}
 \log &\frac{\left \| D_{ x_n  } \big( a^{t_n} b^s\Pi  c_i^{s_i}u^r\big) ( v_n) \right  \|    }{\left \| D_{ x_n  } \big(   b^s\Pi  c_i^{s_i}u^r\big) ( v_n)\big) \right  \|  }
  =  \log \frac{\left \| D_{ a^{t_n}\cdot x_n  } \big( b^s\Pi  c_i^{s_i} u^{r'}   \big)       \circ  D_{ x_n  } a^{t_n}   \big( v_n\big) \right  \|    }{\left \| D_{ x_n  } \big(   b^s\Pi  c_i^{s_i}u^r\big) ( v_n) \right  \|  }\\
 &\quad \ge \log \|D_{ x_n  } a^{t_n}   \big( v_n\big)\| -   \log \| \big (D_{ x_n  }   b^s\Pi  c_i^{s_i}u^r\big) \| _\fib- \log \| \big( D_{ a^{t_n}\cdot x_n  } \big( b^s\Pi  c_i^{s_i} u^{r'} \big)\big)^{\blue -1} \|_\fib \phantom{\Big\|}\\
 &\quad \ge \log \|D_{ x_n  } a^{t_n}   \big( v_n\big)\| -   \log \| D_{ x_n  } \big(    u^r\big)\| _\fib- \log \| D_{ u^r x_n  } \big(   b^s\Pi  c_i^{s_i} \big)\|_\fib
 \\&\quad \quad \quad  - \log \| \big (D_{ u^{r'}a^{t_n}\cdot x_n  } \big( b^s\Pi  c_i^{s_i}   \big) \big)^{\blue -1}\|_\fib -
 \log \| D_{ a^{t_n}\cdot x_n  } \big(  u^{{\blue -r'}}  \big)\|_\fib
 . \phantom{\Big\|}\\
 \end{aligned}
 \end{equation*}

Observe that both  $u^r\cdot  x_n$ and $ u^{r'}a^{t_n}\cdot x_n$ are contained in a fixed compact subset of  $G/\Gamma$ and hence, by Claim \ref{claim90}, having taken $\delta>0$ sufficiently small in the construction of the \Folner sequence, from the constraints on $s_i$ and $s$ we have  $\|D_{u^{r}x_n}\Pi  c_i^{s_i}b^s\|_\fib\le  e^{\lambda t_n/100}$  and $\|\big(D_{u^{r'}a^{t_n}\cdot x_n}\Pi  c_i^{s_i}b^s\big)^{\blue -1}\|_\fib\le  e^{\lambda t_n/100}$ for all $n$ sufficiently large.

Moreover, from Claim  \ref{siegelunipotent},  we have $\|D_{x_n}u^r \|_\fib\le e^{\lambda t_n/100}$  for all $n$ sufficiently large.

Finally, there exists $\kappa>0$ such that $\|r'\| \le e^{\kappa t_n} \|r\|$ whence $r'\in   B_{\R^{m-1}}(e^{(200+\kappa)t_n})$.  Again   from Claim  \ref{siegelunipotent},
 we have $\|D_{a^{t_n}\cdot x_n}u^{\blue -r'} \|_\fib\le e^{\lambda t_n/100}$   for $n$ sufficiently large.
 Combined with  \eqref{eq:cleansingDeluge} we then have
\[\frac{1}{t_n} \int_{M^\alpha }   {\log  \|\calA(a ^{t_n},\cdot) \|  \ d \xi_n}   \ge \lambda- \frac 4{100} \lambda. \]
 Proposition \ref{positiveexponent} then follows from  Lemma \ref{lemma:firstexponents}.
\end{proof}

\subsection{Proof of Proposition \ref{prop:maybeweshouldstatethemainresultatsomepoint}}
Having   assumed that $\chi_{\mathrm{max}}$ in \eqref{eq:fanorlamps} is non-zero, we arrive at a contradiction.
Take any weak-$*$ subsequential limit $\mu_\infty$ of the sequence of measure  $\{\mu_n\}$ on $M^\alpha$.  We have that $\mu_\infty$ is $A$-invariant and has a non-zero fiberwise Lyapunov exponent for the fiberwise derivative  over the action of $a^t$.
Moreover, we have that $\mu_\infty$ projects to $\nu_\infty$ on $G/\Gamma$ which, as discussed above, is the Haar measure on $G/\Gamma$.
We may replace $\mu_\infty$ with an $A$-ergodic component $\mu'$ with the same properties as above.  Then $\mu$ is $A$-ergodic, projects to Haar, and the fiberwise derivative cocycle over the $A$-action on $(M^\alpha, \mu)$ has a non-zero Lyapunov exponent functional $\lambda_i\colon A\to \R.  $

 As in the conclusion of Lemma \ref{uni1},   the arguments of \cite[Section 5.5]{BFH} using
\cite[Proposition 5.1]{AWBFRHZW-latticemeasure}  imply that the measure  $\mu$ is, in fact, $\Sl(m,\R)$-invariant. As before, we note that \cite[Proposition 5.1]{AWBFRHZW-latticemeasure} does not assume $\Gamma$ is cocompact, so the algebraic argument applying that proposition in \cite[Section 5.5]{BFH} goes through verbatim.   For a more self-contained proof that applies since we only consider the case of $\SL(m,\R)$ see \cite[Proposition 4]{BDZ}.
We then obtain a contradiction with Zimmer's cocycle superrigidity by constraints on the dimension of the fibers of $M^\alpha$.
Thus we must have $\chi_{\mathrm{max}}=0$ and Proposition \ref{prop:maybeweshouldstatethemainresultatsomepoint} follows.

  \bibliographystyle{AWBmath}

\bibliography{bibliography}

\begin{thebibliography}{BRHW}

\bibitem[Ath]{MR2247652}
J.~S. Athreya.
\newblock {\em {Quantitative recurrence and large deviations for {T}eichmuller
  geodesic flow}}, Geom. Dedicata {\bf 119}(2006), 121--140.

\bibitem[Bil]{MR1700749}
P.~Billingsley.
\newblock {\em Convergence of probability measures}.
\newblock Wiley Series in Probability and Statistics: Probability and
  Statistics. John Wiley \& Sons, Inc., New York, second edition, 1999.
\newblock {\em A Wiley-Interscience Publication}.

\bibitem[BDZ]{BDZ}
A.~Brown, D.~Damjanovic, and Z.~Zhang.
\newblock {\em {$C^1$ actions on manifolds by lattices in Lie groups}}.
\newblock Preprint (2018).
\newblock arXiv:1801.04009.

\bibitem[BFH]{BFH}
A.~Brown, D.~Fisher, and S.~Hurtado.
\newblock {\em {Zimmer's conjecture: Subexponential growth, measure rigidity,
  and strong property (T)}}, Preprint (2016).
\newblock arXiv:1608.04995.

\bibitem[BRH]{AWB-GLY-P1}
A.~Brown and F.~Rodriguez~Hertz.
\newblock {\em {Smooth ergodic theory of {$\Z^d$}-actions part 1: {L}yapunov
  exponents, dynamical charts, and coarse {L}yapunov manifolds}}.
\newblock Preprint (2016).
\newblock arXiv:1610.09997.

\bibitem[BRHW]{AWBFRHZW-latticemeasure}
A.~Brown, F.~Rodriguez~Hertz, and Z.~Wang.
\newblock {\em {Invariant measures and measurable projective factors for
  actions of higher-rank lattices on manifolds}}.
\newblock Preprint (2016).
\newblock arXiv:1609.05565.

\bibitem[BM]{MR1911660}
M.~Burger and N.~Monod.
\newblock {\em {Continuous bounded cohomology and applications to rigidity
  theory}}, Geom. Funct. Anal. {\bf 12}(2002), 219--280.

\bibitem[dlS]{delaSallenonuniform}
M.~de~la Salle.
\newblock {\em {Strong {(T)} for higher rank lattices}}.
\newblock (2017).
\newblock arXiv:1711.01900.

\bibitem[EM1]{MR1230290}
A.~Eskin and C.~McMullen.
\newblock {\em {Mixing, counting, and equidistribution in {L}ie groups}}, Duke
  Math. J. {\bf 71}(1993), 181--209.

\bibitem[EM2]{MR2787598}
A.~Eskin and M.~Mirzakhani.
\newblock {\em {Counting closed geodesics in moduli space}}, J. Mod. Dyn. {\bf
  5}(2011), 71--105.

\bibitem[FM]{MR2039990}
D.~Fisher and G.~A. Margulis.
\newblock {Local rigidity for cocycles}.
\newblock In {\em Surveys in differential geometry, {V}ol.\ {VIII} ({B}oston,
  {MA}, 2002)}, volume~8 of {\em Surv. Differ. Geom.}, pages 191--234. Int.
  Press, Somerville, MA, 2003.

\bibitem[FH]{MR2219247}
J.~Franks and M.~Handel.
\newblock {\em {Distortion elements in group actions on surfaces}}, Duke Math.
  J. {\bf 131}(2006), 441--468.

\bibitem[Ghy]{MR1703323}
{\'E}.~Ghys.
\newblock {\em {Actions de r\'eseaux sur le cercle}}, Invent. Math. {\bf
  137}(1999), 199--231.

\bibitem[KKLM]{KKLM}
S.~Kadyrov, D.~Y. Kleinbock, E.~Lindenstrauss, and G.~A. Margulis.
\newblock {\em {Singular systems of linear forms and non-escape of mass in the
  space of lattices}}.
\newblock Preprint (2016).
\newblock arXiv:1407.5310.

\bibitem[KM]{MR1652916}
D.~Y. Kleinbock and G.~A. Margulis.
\newblock {\em {Flows on homogeneous spaces and {D}iophantine approximation on
  manifolds}}, Ann. of Math. (2) {\bf 148}(1998), 339--360.

\bibitem[Kle]{MR2434296}
D.~Kleinbock.
\newblock {\em {An extension of quantitative nondivergence and applications to
  {D}iophantine exponents}}, Trans. Amer. Math. Soc. {\bf 360}(2008),
  6497--6523.

\bibitem[LMR1]{MR1244421}
A.~Lubotzky, S.~Mozes, and M.~S. Raghunathan.
\newblock {\em {Cyclic subgroups of exponential growth and metrics on discrete
  groups}}, C. R. Acad. Sci. Paris S\'er. I Math. {\bf 317}(1993), 735--740.

\bibitem[LMR2]{MR1828742}
A.~Lubotzky, S.~Mozes, and M.~S. Raghunathan.
\newblock {\em {The word and {R}iemannian metrics on lattices of semisimple
  groups}}, Inst. Hautes \'Etudes Sci. Publ. Math. (2000), 5--53 (2001).

\bibitem[Mar1]{MR0484767}
G.~A. Margulis.
\newblock {\em {Explicit constructions of expanders}}, Problemy Pereda\v ci
  Informacii {\bf 9}(1973), 71--80.

\bibitem[Mar2]{MR1754775}
G.~Margulis.
\newblock {Problems and conjectures in rigidity theory}.
\newblock In {\em Mathematics: frontiers and perspectives}, pages 161--174.
  Amer. Math. Soc., Providence, RI, 2000.

\bibitem[Pol]{MR1946555}
L.~Polterovich.
\newblock {\em {Growth of maps, distortion in groups and symplectic geometry}},
  Invent. Math. {\bf 150}(2002), 655--686.

\bibitem[Rat1]{MR1262705}
M.~Ratner.
\newblock {\em {Invariant measures and orbit closures for unipotent actions on
  homogeneous spaces}}, Geom. Funct. Anal. {\bf 4}(1994), 236--257.

\bibitem[Rat2]{MR1135878}
M.~Ratner.
\newblock {\em {On {R}aghunathan's measure conjecture}}, Ann. of Math. (2) {\bf
  134}(1991), 545--607.

\bibitem[Sha1]{MR1291701}
N.~A. Shah.
\newblock {\em {Limit distributions of polynomial trajectories on homogeneous
  spaces}}, Duke Math. J. {\bf 75}(1994), 711--732.

\bibitem[Sha2]{MR1767270}
Y.~Shalom.
\newblock {\em {Rigidity of commensurators and irreducible lattices}}, Invent.
  Math. {\bf 141}(2000), 1--54.

\bibitem[Wit]{MR1198459}
D.~Witte.
\newblock {\em {Arithmetic groups of higher {${\bf Q}$}-rank cannot act on
  {$1$}-manifolds}}, Proc. Amer. Math. Soc. {\bf 122}(1994), 333--340.

\end{thebibliography}

\end{document}